\documentclass[12pt,a4paper,reqno]{amsart}
\usepackage{amsmath,amssymb}
\usepackage[top=3.5cm, bottom=3.5cm, outer=3.5cm, inner=3.5cm, heightrounded, marginparwidth=1.9cm, marginparsep=0.3cm]{geometry}

\usepackage{graphics,epsfig,color,psfrag,dsfont}
\usepackage[percent]{overpic}
\usepackage[mathscr]{euscript}
\theoremstyle{plain}
\newtheorem{theorem}{Theorem}[section]
\newtheorem{proposition}[theorem]{Proposition}
\newtheorem{lemma}[theorem]{Lemma}

\newtheorem{claim}[theorem]{Claim}
\theoremstyle{definition}

\theoremstyle{remark}
\newtheorem{remark}[theorem]{Remark}

\numberwithin{equation}{section}
\numberwithin{theorem}{section}
\newcommand{\mc}[1]{{\mathcal #1}}
\newcommand{\bb}[1]{{\mathbb #1}}

\newcommand{\upbar}[1]{\,\overline{\! #1}}

\renewcommand{\epsilon}{\varepsilon}
\renewcommand{\tilde}{\widetilde}
\renewcommand{\hat}{\widehat}

\renewcommand{\div}{\mathop{\rm div}\nolimits}

\definecolor{light}{gray}{.9}

 \definecolor{refkey}{gray}{.4}
 \definecolor{labelkey}{gray}{.4}

\title[Flows, currents, and cycles]
{Flows, currents, and cycles for  Markov Chains: large 
deviation asymptotics}

\author[L.\ Bertini]{Lorenzo Bertini}
\address{Lorenzo Bertini \hfill\break \indent
   Dipartimento di Matematica, Universit\`a di Roma `La Sapienza'
   \hfill\break \indent
   P.le Aldo Moro 2, 00185 Roma, Italy}
 \email{bertini@mat.uniroma1.it}

\author[A.\ Faggionato]{Alessandra Faggionato}
\address{Alessandra Faggionato \hfill\break \indent
  Dipartimento di Matematica, Universit\`a di Roma `La Sapienza'
  \hfill\break \indent
  P.le Aldo Moro 2, 00185 Roma, Italy}
\email{faggiona@mat.uniroma1.it}

\author[D.\ Gabrielli]{Davide Gabrielli}
\address{Davide Gabrielli \hfill\break \indent
  DISIM, Universit\`a dell'Aquila
  \hfill\break\indent
  Via Vetoio,   67100 Coppito, L'Aquila, Italy}
\email{gabriell@univaq.it}

\begin{document}

\begin{abstract} 
We consider a continuous time Markov chain on a countable state space. We  prove a joint large deviation principle (LDP)  of the empirical measure and  current
in the limit of large time interval. The proof is based on  results on the joint large deviations of the empirical measure and    flow obtained in \cite{BFG}. By improving such results we also show, under additional assumptions, that the LDP holds with the strong $L^1$  topology on  the space of currents.  
We deduce a general version of the Gallavotti--Cohen (GC)  symmetry  for the current field and show that it implies the so--called fluctuation theorem 
for the GC functional. We also analyze the large deviation properties  of  generalized  empirical currents associated to a fundamental basis in the cycle space, which, as we show, are given by  the first class homological coefficients in the   graph underlying the Markov chain. 
Finally, we discuss in detail some examples.

\bigskip


\smallskip


\end{abstract}

\maketitle
\thispagestyle{empty}

\section{Introduction}
\label{introduzione}

We consider a continuous time Markov chain on a countable (finite or
infinite) state space $V$ with transition rates $r(\cdot,\cdot)$. We
assume that the chain is ergodic and positive recurrent, so that it
admits a unique invariant probability distribution $\pi$.

  A
natural observable is given by  the \emph{empirical measure} $\mu_T$, which
accounts for the fraction of time spent on the variouos states up to
time $T$. As $T\to \infty$, $\mu_T$ converges to $\pi$. The large
deviation principle for the family $\{\mu_T\}$ is  the classical
Donsker-Varadhan theorem \cite{DV4}.  
Other natural observables are the \emph{empirical flow} $Q_T$ and
\emph{empirical current} $J_T$, which respectively account for the
total numbers of jumps and for the net flow between pairs of states
per unit of time. In particular, given two states $y,z\in V$, it holds  
$J_T(y,z)=Q_T(y,z)-Q_T(z,y)$.  As $T\to \infty$, $Q_T(y,z)$ and
$J_T(y,z)$ respectively converge to $\pi(y)r(y,z)$ and
$\pi(y)r(y,z)-\pi(z)r(z,y)$.  The large deviation principle for the
family $\{(\mu_T,Q_T)\}$ is proven in \cite{BFG}.

\medskip

The interest for these observables comes from several
applications.  We mention some of them, mainly related to the concept of work, to the Gallavotti--Cohen functional and to the concept of activity in kinetically constrained spin systems. 

 When the Markov chain models the stochastic dynamics of
a physical particle in presence of an external field and thermal noise,
the work done by the field can be expressed in terms of the
empirical current. When modeling  biochemical systems, the
state describes both the mechanical and the chemical configuration. One
is then interested on the work done both by the applied mechanical
force and the chemical one, in which the latter is induced by
differences in the chemical potentials. In both cases the work is a
linear function of the empirical current. 
Significant examples are biochemical systems given by single molecules
like molecular motors \cite{S}.

In out--of--equilibrium statistical mechanics a much studied
observable is the Gallavotti--Cohen functional $W_T$. It is defined
as follows \cite{LS}: $e^{-T W_T}$ 
is the Radon--Nikodym derivative of the time--reversed stationary process 
$\bb P_\pi^*$ 
w.r.t. the stationary process itself $\bb P_\pi$ in the time window
$[0,T]$. 
It follows that $W_T$ accounts for the
irreversibility of the stochastic dynamics and its expectation
w.r.t.\ $\bb P_\pi$ is 
the  relative entropy of $\bb P_\pi$ w.r.t.\ 
$\bb P_\pi^*$ per unit of time. 
By a straightforward computation, it turns out that $W_T$ 
is a linear function of the empirical current apart boundary terms. 
When the state space is finite, the large deviation principle for
$\{W_T\}$ has been derived in \cite{LS} by the G\"artner-Ellis
theorem. The so-called \emph{fluctuation theorem} (or
\emph{Gallavotti-Cohen symmetry})  is
then the identity $\iota(u)-\iota(-u)= -u$ satisfied by  the corresponding rate
function $\iota\colon \bb R \to \bb R_+$.   

For kinetically constrained spin systems, see  \cite{BT} and references
therein, the empirical flow is a relevant observable and its large deviation
properties exhibit peculiar and rich features. 
 More precisely,   given a system of $N$ spins,  
 the $N$--normalized  total number of jumps per unit time (also  called \emph{activity}) has a nontrivial second order LDP in the limit  $T \to \infty$ and afterwards $N \to \infty$. We point out that the above activity  is proportional 
 to   the total mass of the empirical flow.

\medskip
Starting from the results in \cite{BFG}, in this paper we derive the
large deviation principle for the family $\{(\mu_T,J_T)\}$ 
(Theorem~\ref{LDP:misura+cor}). By contraction, we then deduce the
large deviation principle for the Gallavotti-Cohen functional and show
the rate function $\iota$ satisfies the Gallavotti-Cohen symmetry 
(Theorem~\ref{LDP:GC}). We remark that this derivation yields an
explicit variational representation of $\iota$, while the derivation
via G\"artner--Ellis theorem gives a spectral characterization \cite{LS}.   
For infinite state spaces  there are however some technical issues that are best exemplified in
the case of a single particle performing a random walk on $\bb Z^d$
with confining potential $U$ and external field $F$. Since the result
in \cite{BFG} is proven by using the bounded weak* topology for the
empirical flow, the contraction can be performed only when the
external field vanishes at infinity. On the other hand, a natural
condition is that $F$ is bounded. To overcome the requirement of $F$
vanishing at infinity, we prove
the large deviation principle for $\{(\mu_T,Q_T)\}$ in the
strong $L^1$ topology for the empirical flow 
(Theorem~\ref{freedom}) under (needed) additional conditions in the
general setting.  As further reinforcement of the results of \cite{BFG} we also show  that some  
 technical assumption there   can be dropped (see Proposition \ref{bresaola}).

We continue  our investigation of  Gallavotti--Cohen type symmetries.
Consider the transition graph $G$, with vertex set $V$, of the Markov
chain. 
For biochemical models, as explained in Section \ref{sec:hom}, the work of the mechanical/chemical forces can
be expressed in terms of the homological coefficients of the trajectory
in a suitable basis of the first cellular homology class $H_1(G;\bb R)$
of $G$ \cite{SCH}.  For finite state space, the analysis of the large
deviations of the homological coefficients and the related
Gallavotti--Cohen symmetry has been deduced in \cite{AG1,FD} via G\"artner--Ellis theorem. We
extend this result to infinite state space emphasizing the
relationship of the homological coefficients with the empirical
current (Theorem~\ref{anonimo}). 
We finally point out that  the Gallavotti-Cohen symmetry both for
the Gallavotti-Cohen functional and the homological coefficients is a 
consequence of a general symmetry of the rate functional for 
$\{(\mu_T,J_T)\}$ (Theorem~\ref{iacopo_ale}).

Finally, in Section \ref{s:ex} we discuss several  examples in which some rate
functionals can be computed explicitly.   

We conclude with further  bibliographical  remarks. In the context of finite state space, the joint LDPs for $\{(\mu_T, Q_T)\}$ and   $\{(\mu_T, J_T)\}$  have been discussed in \cite{WK,Maes2,Maes3}. See also \cite{Maes4} for a perturbative expansion in the context of non--equilibrium statistical mechanics. For countable state spaces, a weak form of joint LDP for $\{(\mu_T, Q_T)\}$ is  derived in \cite{dF}. 
The joint LDP for  $\{(\mu_T, J_T)\}$  of a Brownian motion on a compact Riemannian manifold is proved in \cite{KKT,K}. See also the discussion in \cite{Maes1} for diffusions on the  torus $\bb T^d$ and on $\bb R^d$ with a confining potential.

\section{Basic setting}
\label{definizioni}

We consider a continuous time Markov chain $\xi_t$, $t \in \bb R_+$
on a countable (finite or infinite) state space $V$. The Markov
chain is defined in terms of the \emph{jump rates} $r(x,y)$, $x \not
=y$ in $V$, from which one derives the holding times and the jump
chain \cite{N}. %
Since the holding time at $x\in V$ is

The basic assumptions on the chain are the following:
\begin{itemize}
\item[(A1)]
  for each $x \in V$, $r(x):= \sum _{y \in V} r(x,y)$ is finite;
\item[(A2)]
   for each $x \in V$ the Markov chain $\xi^x_t$ starting from $x$
   has no explosion a.s.;
\item[(A3)]
  the Markov chain is irreducible, i.e.\ for each $x,y \in V$ and $t>0$
  the event $\{\xi^x_t=y\}$ has strictly positive probability;
\item[(A4)]
  there exists a unique invariant probability measure, that is denoted
  by $\pi$.
\end{itemize}
By assumption (A1) the holding time at $x \in V$ is a well defined exponential random variable of parameter
 $r(x)$.
As in \cite{N}, by invariant probability measure $\pi$ we mean a
probability measure on $V$ such that
\begin{equation}
  \label{invariante}
  \sum _{y \in V} \pi(x) \, r(x,y)
  = \sum _{y \in V} \pi(y) \, r(y,x)\qquad \forall \:x \in V
\end{equation}
where we understand $r(x,x)=0$.
We refer to 
Section \ref{ragnatela} for a discussion on the above  assumptions (A1),...,(A4)  and their relation with Condition $C(\sigma)$ introduced in the next section.
We only recall that   $\pi(x)>0$ for all $x \in V$, the Markov chain
starting with distribution $\pi$ is stationary (i.e.\ is left
invariant by time-translations), and   the ergodic theorem holds,
i.e.\ for any bounded function $f: V \to \bb R$ and any initial
distribution
\begin{equation}
  \label{ergodico}
  \lim_{T \to +\infty} \frac{1}{T} \int_0^T\!dt\, f(\xi_t)
  = \langle\pi, f\rangle
  \qquad \textrm{a.s.}
\end{equation}
where $\langle\pi,f\rangle$ denotes the expectation of $f$ with
respect to $\pi$.

We consider $V$ endowed with the discrete topology and the
associated Borel $\sigma$-algebra given by the collection of all the
subsets of $V$.  Given $x\in V$, the distribution of the Markov
chain $\xi^x_t$ starting from $x$, is a probability measure on the
Skorohod space of c\`adl\`ag paths $D(\bb R_+;V)$ that we denote by $\bb P_x$. The
expectation with respect to $\bb P_x$ is denoted by $\bb E_x$. In
the sequel we consider $D(\bb R_+;V)$ equipped with the canonical
filtration,
 the canonical coordinate in $D(\bb R_+;V)$ is denoted by $X_t$.
The set of probability measures on $V$ is denoted by $\mc P(V)$ and
it is considered endowed with the topology of weak convergence and
the associated Borel $\sigma$-algebra.

\subsection{Empirical measure and empirical flow}
\label{s:emef}

Given $T>0$ the \emph{empirical measure}
$\mu_T\colon D(\bb R_+;V)\to \mc P(V)$ is defined by
\begin{equation*}
  \mu_T \, (X) = \frac 1T\int_0^T\!dt \, \delta_{X_t}\,,
\end{equation*}
where $\delta_y$ denotes the pointmass at $y$.
By the ergodic theorem    the sequence of probabilities $\{\bb P_x \circ
\mu_T^{-1}\}_{T>0}$ on $\mc P(V)$ converges to $\delta_\pi$.

\smallskip

 We denote by  $E$  the (countable) set of ordered
edges in $V$ with strictly positive jump rate, i.e.\
\[ E:= \{ (y,z)\in V \times V \,:\, r(y,z)>0\}\,,\]  by $L^1(E)$ the
collection of absolutely summable functions on $E$ and by $\| \cdot
\|$ the associated $L^1$--norm. The set of positive elements in
$L^1(E)$ is denoted by $L^1_+(E)$.
Note that, since $V$ has the discrete topology and is countable, any path in
$D(\bb R_+;V)$ has a locally finite number of jumps. In particular,
for each $T>0$ we can
define the \emph{empirical flow} as the map $Q_T
\colon D(\bb R_+;V)\to L^1_+(E)$ given by
\begin{equation}
  \label{montecarlo}
  Q_T(y,z) \, (X) :=
  \frac{1}{T} \sum_{0\leq t\leq  T} \mathds{1}\left( X_{t^-}=y,\; X_{t}=z\right)
  \qquad (y,z)\in E\,,
\end{equation}
where, in general,  $\mathds{1}(A)$ denotes the characteristic function of $A$.
Namely, $ T Q_T(y,z)$ gives    the  number of jumps from
$y$ to $z$ in the time interval $[0,T]$.

Elements of $L^1_+(E)$ will be denoted by $Q$ and called
\emph{flows}. Given a flow $Q$ we let its \emph{divergence}
$\div Q \colon V\to \bb R$ be the pointwise difference between the outgoing flow and the ingoing one, namely
\begin{equation}
  \label{divergenza_fluss}
  \div Q \, (y)= \sum _{z : \, (y,z)\in E} Q(y,z)- \sum_{z:\, (z,y)\in E} Q(z,y),
  \qquad y\in V.
\end{equation}
 Observe that the divergence maps $L^1_+(E)$  to $L^1(V)$.
To each probability $\mu \in \mc P(V)$  such that $\langle\mu,
r\rangle<+\infty$ we associate the flow $Q^\mu$ defined by
\begin{equation}
  \label{Qmu}
  Q^\mu(y,z) := \mu(y) \, r(y,z)
  \qquad (y,z)\in E.
\end{equation}
Note that  $Q^\mu$ has vanishing divergence if  and only if $\mu$ is
invariant, i.e. $\mu=\pi$.

By the ergodic theorem and a martingale argument (cf. \cite{BFG}) one can show  that
for each $x\in V$ and $(y,z)\in E$ the sequence of real random
variables $Q_T(y,z)$ converges as $T\to+\infty$ to $Q^\pi(y,z)$  in probability with respect to $\bb
P_x$. 

\section{Joint large deviations for the empirical measure and flow}
\label{s:ldef}

In this section we recall the main results of
 \cite{BFG}.
  The space  $L^1_+(E)$  is  endowed
   with  the bounded weak* topology \cite{Me}, which is defined as follows. Let $C_0(E)$ be the space of functions $f:E \to \bb R$ vanishing  at infinity, endowed with the uniform norm. Then its dual space is given by $L^1(E)$ endowed with the strong topology (i.e. the topology determined by the $L^1$--norm).  A basis of the bounded weak* topology on $L^1(E)$ is then given by the sets
\[ \{ q \in L^1(E):  \langle  q-\bar q ,f _n \rangle  <1  \; \;\forall n \geq 1  \}\]
as $\bar q$ varies among $L^1(E)$ and  $(f_n)_{n \geq 1}$ varies among the sequences in $C_0(E)$ converging to $0$ in uniform norm.
In general, given $q \in L^1(E)$ and $f \in C_0(E)$, we set $\langle q , f \rangle:= \sum _{e \in E} q(e) f(e) $. Finally, the
 bounded weak* topology on $L^1_+(E)$ is the inherited subspace topology  on $L^1_+(E)\subset L^1(E)$, when
$L^1(E)$ itself is endowed with the  above defined  bounded weak* topology.

 One can prove (cf. \cite{Me}[Cor. 2.7.4])  that
 a  subset $W \subset L^1 (E)$ is open in the bounded weak* topology
 if and only if for each $\ell >0$ the
set $\{ q \in W\,:\, \|q\|_1 \leq  \ell\}$ is open in the ball  $\{ q \in L^1(E)\,:\, \|
q\|_1 \leq \ell\}$ endowed with  the weak*  topology inherited from $L^1(E)$. We recall that the weak*  topology of $L^1(E)$ is the weakest topology such that the map $L^1(E)\ni q \to \langle q,f \rangle   \in \bb R$ is continuous for any map $f \in C_0(E)$.
  When $E$ is finite, the bounded weak* topology coincides with the  strong topology. If $E$ is infinite then
 the former is weaker than the latter and  cannot be metrized.

\medskip

We can now recall the LDP proved in \cite{BFG}. We start from the assumptions.
To this aim,
 given $f\colon V\to \bb R$ such that  $\sum_{y\in V }r(x,y) \, |f(y)| <+\infty$
for each $x\in V$,
 we
denote by $Lf \colon V\to \bb R$ the function defined by
\begin{equation}
  \label{Lf}
  L f\, (x) := \sum_{y\in V} r(x,y) \big[ f(y)-f(x)\big]
  ,\qquad x\in V.
\end{equation}

\noindent
{\bf Condition $\mathbf{C(\sigma)}$}
\emph{
 Given $\sigma \in \bb R_+$ we say that Condition $C(\sigma)$ holds if
  there exists a sequence of functions $u_n\colon V \to (0,+\infty)$
  satisfying the following requirements:
  \begin{itemize}
  \item [(i)] For each $x\in V$ and $n\in\bb N$ it holds $\sum_{y\in V} r(x,y)
    u_n(y) <+\infty$. 
  \item [(ii)] The sequence $u_n$ is uniformly bounded from below.
    Namely, there exists $c>0$ such that $u_n(x)\ge c$ for any $x\in
    V$ and $n\in\bb N$.
  \item[(iii)] The sequence $u_n$ is uniformly bounded from above on compacts.
    Namely, for each $x\in V$ there exists a
    constant $C_x$ such that for any $n\in\bb N$ it holds $u_n(x)\le C_x$.
  \item [(iv)] Set $v_n :=  - Lu_n / u_n$. The sequence
    $v_n\colon V\to \bb R$ converges pointwise to some $v\colon V\to \bb R$.
  \item[(v)] The function $v$ has compact level sets.
    Namely, for each $\ell\in \bb R$
    the level set $\big\{x \in V \,:\, v(x)\leq \ell\big\}$ is
    finite.
  \item[(vi)]
    There exists 
     a positive
    constant $C$ such that
    $v \ge \sigma \, r - C$.
  \end{itemize}
}

\medskip
 Let
$\Phi\colon \bb R_+ \times \bb R_+ \to [0,+\infty]$ be the function
defined by
\begin{equation}
  \label{Phi}
   \Phi (q,p)
   :=
   \begin{cases}
     \displaystyle{ q \log \frac qp - (q-p)}
     & \textrm{if $q,p\in (0,+\infty)$}
     \\
     \;p  & \textrm{if $q=0$, $p\in [0,+\infty)$}\\
     \; +\infty & \textrm{if $p=0$ and $q\in (0,+\infty)$.}
   \end{cases}
\end{equation}
For $p>0$, $\Phi( \cdot, p)$ is a nonnegative strincly convex function and is zero only at $q=p$. Indeed, since $\Phi(q,p)=\sup_{s\in \mathbb R}\left\{qs-p(e^s-1)\right\}$, $\Phi$   is the rate function for the LDP of the sequence  $N_T/T$ as $T \to +\infty$, $(N_t)_{t \in \bb R_+}$
being a Poisson process with parameter $p$.

 Finally,
we let $I\colon \mc P(V)\times L^1_+(E) \to
[0,+\infty]$ be the functional defined by
\begin{equation}
  \label{rfq}
  I (\mu,Q) :=
  \begin{cases}
    \displaystyle{
    \sum_{(y,z)\in E} \Phi \big( Q(y,z),Q^\mu(y,z) \big)
    }& \textrm{if  } \; \div Q =0\,,\; \langle \mu,r \rangle < +\infty
    \\
    \; +\infty  & \textrm{otherwise}.
  \end{cases}
\end{equation}
\begin{remark}\label{silente} As proved in \cite{BFG}[Appendix A]    the above condition
 $\langle \mu, r \rangle < +\infty$ can be removed, since
 the series  in \eqref{rfq} diverges  if $\langle \mu , r \rangle
 =+\infty$.
\end{remark}

\begin{theorem}[Bertini, Faggionato, Gabrielli \cite{BFG}]$\,$
  \label{LDP:misura+flusso} \\
  Endow $\mc P(V)$ with the weak topology and $L^1_+(E)$ with  the bounded weak* topology.
   Assume the Markov chain satisfies (A1)--(A4) and   Condition $C(\sigma)$ with $\sigma >0$.
  Then, as $T\to+\infty$, the sequence of probability measures $\{\bb
  P_x\circ (\mu_T,Q_T)^{-1}\}$ on $\mc P(V)\times L^1_+(E)$
  satisfies a LDP with good and convex rate function $I$.
  Namely, for   each closed set $\mc C\subset \mc P(V)\times L^1_+(E)$,
  and each open  set $\mc A \subset \mc P(V)\times L^1_+(E)$, it
  holds for each $x \in V$
  \begin{align}
    \label{ubldp}
    & \varlimsup_{T\to+\infty}\;
    \frac 1T \log \bb  P_x \Big( (\mu_T,Q_T) \in \mc C \Big)
    \le -\inf_{(\mu,Q)\in \mc C} I(\mu,Q),
    \\
    \label{lbldp}
    & \varliminf_{T\to+\infty}\;
    \frac 1T \log \bb P_x \Big( (\mu_T,Q_T) \in \mc A \Big)
    \ge -\inf_{(\mu,Q)\in \mc A} I(\mu,Q).
   \end{align}
\end{theorem}
 We point out that Condition $C(\sigma)$ with $\sigma>0$
implies that $\langle\pi ,r\rangle<+\infty$ (cf. \cite{BFG}[Lemma 3.9]) and that  $r(\cdot)$ has compact level sets (cf. \cite{BFG}[Remark 2.3]). Moreover,
Condition $C(0)$ (i.e. $C(\sigma)$ with $\sigma=0$) with (i) replaced by the fact that $u_n $ belongs to the domain of the infinitesimal generator of the Markov chain $(\xi_t)_{t \in \bb R_+}$, and with $Lu_n$ defined as the infinitesimal generator applied to $u_n$, is the condition under which the large deviation of the empirical measure is derived in \cite{DV4}--(IV). Finally, see \cite{BFG}[Section 2.3], it holds $I(\mu,Q)=0$ if and only if $\mu= \pi$ and $Q= Q^\pi$ and Theorem \ref{LDP:misura+flusso}  implies that the empirical flow $Q_T$, sampled  according to $\bb P_x$, converges    to $Q^\pi $ in $L^1_+(E)$ (endowed of the bounded weak* topology).

\begin{remark}
As discussed in \cite{BFG} Theorem \ref{LDP:misura+flusso} holds also replacing Condition $C(\sigma)$, $\sigma>0$, with a suitable hypercontractivity assumption (see Condition 2.4 there). Also the results we present in the rest of the present article could be obtained under this alternative assumption.
\end{remark}

\section{Comments on the main assumptions}\label{ragnatela}

We first recall some basic facts from
\cite{N}[Chapter 3]. 
Assuming (A1) and irreducibility (A3), assumptions (A2) and (A4)
together are equivalent to the fact that all states are positive
recurrent.
  In (A4) one could remove the assumption of
uniqueness of the invariant probability measure, since for an
irreducible Markov chain there can be at most  one.
 We observe that if $V$ is finite then
(A1) and (A2) are automatically satisfied, while (A3) implies (A4).

\begin{proposition}\label{bresaola}
If the Markov chain satisfies assumptions (A1), (A2), (A3) and Condition $C( \sigma)$ for some $\sigma \geq 0$, then   (A4) is verified.
\end{proposition}
\begin{proof}
The core of the proof will consist in showing  that there exists a probability measure $\pi $ on $V$ such that $\pi P(s) = \pi$ for some $s>0$, where $P(s) $ is the $V\times V$--matrix such that $P_{y,z}(s)= \bb P_y(\xi_s=z)$.

Before proving this property, let us explain how to deduce  that $\pi$ is invariant in the algebraic sense \eqref{invariante} (as already stressed, uniqueness in (A4) is a consequence of (A3)). Due to \cite{N}[Th. 3.5.5] we only need to prove that the Markov chain $\xi$ is recurrent. To this aim, consider the discrete time Markov chain $\zeta_n:= \xi_{n s} $ with associated stochastic
matrix $P(s)$. Note that the irreducibility of $\xi$ implies the irreducibility of $\zeta$ and that the  condition $\pi P(s)=\pi$ corresponds to the fact that $\pi$ is an invariant distribution for $\zeta$. Hence,
 due to \cite{N}[Th. 1.7.7],  each state is positive recurrent  for the Markov chain $\zeta$ and therefore is recurrent for the Markov chain $\xi$.

   It remains to exhibit $\pi \in \mathcal{P}(V)$ such that $\pi P(s) = \pi$.
 To this aim, we fix $x \in V$ and,  given an integer $n \geq 1$, we define $\pi _n  \in \mc P (V)$ as  $\pi_n (A)= \bb E_x ( \mu_n (A) )$ for all $A \subset V$ ($\mu_n$ denotes the empirical measure at time $n$). We claim that, due to Condition $C(\sigma)$, the sequence $\{\pi_n\}_{n \geq 1}$ is tight in $\mc P(V)$.   In the proof of Proposition  3.6 in \cite{BFG} we have deduced (without using (A4)) that for each $\ell \geq 1$ there exists a finite set $K_\ell \subset V$ such that  $\lim _{n \to \infty} \bb P_x\left( \mu_n(K_\ell^c) > \frac{1}{\ell}\right )=0$. Since
 \begin{equation*}
 \begin{split} \pi_n (K_\ell^c) =  \bb E_x ( \mu_n (K_\ell^c) )& \leq \frac{1}{\ell} \bb P_x \left( \mu_n (K_\ell^c) \leq \frac{1}{\ell}\right)+  \bb P_x \left( \mu_n (K_\ell^c) > \frac{1}{\ell}\right)\\
 & \leq \frac{1}{\ell} + \bb P_x \left( \mu_n (K_\ell^c) > \frac{1}{\ell}\right)\,,
 \end{split}
 \end{equation*}
 it is simple to obtain that the sequence $\{\pi_n\}_{n \geq 1}$ is tight in $\mc P(V)$. By
  Prohorov theorem (cf. \cite{Bi}[Theorem 5.1])  the sequence is relatively compact, and therefore there exists a subsequence   $n_k \nearrow \infty$ and a probability measure $\pi$ in $\mc P(V)$ such that $\pi_{n_k}$ converges weakly to $\pi$. Let us show that for any $s>0$ it holds $\pi P(s)= \pi$. To this aim we show that $\langle \pi, P(s) f \rangle= \langle \pi, f \rangle$ for any  bounded function $f : V \to \bb R$.
 Since $P(s) f : V \to \bb R$ is bounded and continuous, by the weak convergence we can write
 \begin{equation}\label{flauto} \langle \pi, P(s) f \rangle= \lim _{k \to \infty}  \langle \pi_{n_k
}, P(s) f \rangle\,.
\end{equation}
 On the other hand, given $g:V \to \bb R$ bounded it holds
\[
 \pi_{n} (g) =\bb E_x\left(  \frac{1}{n}\int_0^n g(X_u)du \right)=  \frac{1}{n}\int_0^n \bb E_x \left( g(X_u) \right)du=  \frac{1}{n}\int_0^n  \left[P(u) g \right] (x)du\,.
 \]
 In particular, by using  the above identity twice (both with $g:= P(s)f$ and with r $g:=f$) and using the semigroup property $P(u)P(s)=P(u+s)$, we have
 \begin{equation}\label{giustino}
 \begin{split}
  \langle \pi_{n }, P(s) f \rangle
 &= \langle \pi_n, f \rangle -\frac{1}{n}\int_0^{s} \left[P(u) f\right] (x)du +\frac{1}{n}\int_n^{n+s} \left[P(u) f\right] (x)du\\
 & = \langle \pi_n, f \rangle +O\left(\frac{s}{n}\right) \,.
 \end{split}
\end{equation}
By setting   $n:=n_k$  in \eqref{giustino} and afterwards taking the limit $k\to +\infty$, 
 from the weak convergence of $\pi_{n_k}$ to $\pi$ we conclude that  \eqref{flauto} equals $\langle \pi, f \rangle $.
  \end{proof}


\section{Joint LDP for the empirical measure and flow in the strong $L^1_+(E)$ topology}\label{agostino}

As stated in  Theorem 2.7.2. in \cite{Me},   the bounded weak* topology is weaker than the strong topology in $L^1
_+(E)$, i.e. the one coming from the $L^1$--norm. This means that any  bounded  weakly* open (closed) set is also strongly open (closed).

\begin{proposition}\label{luglio}   Under the same hypotheses of 
 Theorem
 \ref{LDP:misura+flusso}  
   a weak\footnote{By weak joint LDP we mean that  \eqref{ubldp} and \eqref{lbldp} are valid for any $\mc C$ compact and any $\mc A$ open} joint LDP  for $(\mu_T, Q_T)$  holds with  the strong topology on  $L^1_+(E)$.
   \end{proposition}
   \begin{proof} Since any strongly compact subset of $L^1_+(E)$ is bounded weak* compact and therefore bounded weak* closed (as the bounded weak* topology is Hausdorff), the  upper bound for 
   strongly compact subsets is a direct consequence of \eqref{ubldp}.

     On the other hand,  one can verify that the direct proof in \cite{BFG}[Sec. 5]  of  the lower bound \eqref{lbldp} works also for strongly open set. In addition, working with the strong topology,  one has not to require  that each vertex in $V$ is the extreme of only a finite family of edges in $E$  as in \cite{BFG}.
   Indeed, this assumption was necessary in \cite{BFG} to assure that, given a function $\phi :V \to \bb R$ vanishing at infinity and defining $\nabla \phi: E \to \bb R$ as $\nabla \phi (y,z)= \phi(z)- \phi(y)$, then
  the map  
  \begin{equation*}L^1_+(E) \ni Q \mapsto \langle \phi, \div Q \rangle = - \langle \nabla \phi, Q \rangle \in \bb R
  \end{equation*}
  is continuous when  $L^1_+(E)$  is endowed with  the bounded weak* topology. 
  The above map is automatically continuous in the strong topology.
  \end{proof}


We now describe a criterion implying  the (full)   joint  LDP  for $(\mu_T, Q_T)$ when $L^1_+(E)$ is  endowed with the strong topology. To this aim, given $E'\subset E$, we define $Q(E') = \sum _{(y,z) \in E'} Q(y,z) $. Moreover,  fixed a subset $\hat E \subset E$,
we define  the $\hat{E}$--dependent function $H: V \mapsto \bb R$  as
\begin{equation}\label{defH}
H(y):= \frac{ \sum _{z: (y,z)\in \hat E} r(y,z) }{  \sum _{z: (y,z)\in  E} r(y,z)}\,.
\end{equation}
Given $a\in (0,1)$ suppose that $H(y) < a$. Then, after arriving in $y$, the Markov chain has probability $H(y)<a$ to jump from $y$ along an edge in $\hat E$. We then call $a$--\emph{unlikely} all edges $(y,z)$ with $H(y) <a$ and $(y,z) \in \hat E$, while we call $a$--\emph{likely} all edges $(y,z)$ with $H(y) <a$ and $(y,z) \in E \setminus \hat E$.


\begin{theorem}\label{freedom}  Assume  Assumptions (A1), (A2), (A3) and Condition $C(\sigma)$  with $\sigma >0$.  
Suppose  there exists a subset $\hat E \subset E$  such that
\begin{itemize}

\item[(i)] for each $y \in V$ there  exists $z\in V$ with $(y,z) \in \hat E$;

\item[(ii)] the function $H : V \to (0, +\infty)$ defined in \eqref{defH} vanishes at infinity;

\item[(iii)] fixed any  $x \in V$, there exist constants $a_0, \gamma >0$ such that for any $a <a_0$
one can find a subset $W= W(x,a)$ in $E$ satisfying the following properties:

\begin{itemize}
\item[(1)] the complement $E \setminus W$
 is finite;

 \item[(2)] each edge in $W$ is $a$--likely or $a$--unlikely, i.e.\
if $ (y,z) \in W$ then  $ H(y) < a $;
 \item[(3)] for each path exiting from $x$ the number of $a$--unlikely  edges in $W$ is at least $\gamma$--times the total number of edges in $W$. Namely,
  for any path
 $x_1,x_2, \dots, x_n$  with
 $x_1=x$ and $(x_i,x_{i+1} ) \in E$
  it holds
 \begin{equation}\label{broccolo}  \sharp\Big \{ i: (x_i , x_{i+1})  \in \hat E \cap W  \Big\}\, \geq \, \gamma\, \sharp \Big\{ i: (x_i , x_{i+1})  \in W\Big\} \,.
 \end{equation}

 \end{itemize}


\end{itemize}
Then   Theorem \ref{LDP:misura+flusso}  remains valid if $L^1_+(E)$ is endowed with the strong topology instead of the bounded weak* topology.
\end{theorem}

Applications of the above theorem can be found in Proposition \ref{covalent} and in Lemma \ref{claim3} of Section \ref{s:ex}.
We point out that a possible natural candidate for the above set $W$ is given by the set $\{(y,z)\,:\, H(y) < a\}$. In many applications the geometric control of $\{(y,z)\,:\, H(y) < a\}$ is partial, and therefore it can be convenient to use a  subset $W \subset \{(y,z)\,:\, H(y) < a\}$.

\begin{proof} In view of Proposition \ref{luglio},  we only need to prove the exponential tightness of the empirical flow in the strong topology.

The core of the proof consists in showing  that
there exists an invading  sequence of  finite subsets $E_n  \nearrow E$ such that
\begin{equation}\label{sandra}
  \bb P_x \bigl( Q_T( E_n^c) \geq 1/n  \bigr) \leq c_1 e^{-c_2 T n+c_3 T} \qquad \forall T\geq 0\,,\; \forall n \geq 1\,,
\end{equation}
for suitable positive constants $c_1,c_2,c_3$ (depending on $x\in V$).
Let us first derive   from \eqref{sandra} the exponential tightness of the empirical flow in the strong topology. 
 To this aim, fixed positive integers $\ell, m$, we let
$$ \mc K_{m, \ell}:=  \Big\{ Q \in  L^1_+(E)\,:\, \|Q\|\leq \ell \,, \;\; Q(E^c_n) \leq 1/n\; \; \forall n \geq m \Big\} \,.
$$
We claim  that  $\mc K_{m, \ell} \subset L^1_+(E)$ is compact for the strong topology.
Indeed, by Prohorov theorem for measures  \cite{B}[Chapter 8],  the set   $\mc K_{m, \ell}$ is relatively compact in the space of nonnegative finite measures on $E$ endowed with the  weak topology.   Hence, given a sequence $\{Q_k\}_{k \geq 0}$ in $\mc K_{m, \ell}$, at cost to extract a subsequence we can assume that $Q_k $ converges weakly to some $Q: E \to [0,\infty)$ thought of as measure on $E$. By definition of weak convergence (recall that $E$ has the discrete topology, hence any  function on $E$ is continuous) one gets that
$ Q \in \mc K_{m, \ell}$ and that $Q_k(e) \to Q(e)$ for all $e \in E$. In particular, one can estimate $\|Q-Q_k\|\leq 2/n + \sum_{e \in E_n} |Q(e)-Q_k(e)|$ for $n \geq m$. This implies that  $\|Q-Q_k\|$ converges to $0$, hence our claim.

We can bound
\begin{equation} \bb P_x  (Q_T \not \in \mc K_{m, \ell}) \leq \bb P_x \bigl( \|Q_T \| \geq \ell) + \sum _{n \geq m} \bb P_x (Q_T(E^c_n) >1/n  ) \,.
\end{equation}
By Proposition 3.6 in \cite{BFG}  $\lim_{\ell \to +\infty}\varlimsup_{T \to +\infty} \frac{1}{T} \log  \bb P_x \bigl( \|Q_T \| \geq \ell) =-\infty$, while by \eqref{sandra} the series in the above r.h.s. is bounded by $c_1 e^{-c_2 T m+c_3 T}/(1-e^{-c_2T} )$. This  implies the exponential tightness, under $\bb P_x$, of the empirical flow $Q_T$ in $L^1_+(E)$ endowed with the strong topology.

\smallskip
Let us now derive \eqref{sandra}.  We first  point out that, for  suitable positive constants $\lambda$ and $c$, it holds 
\begin{equation}\label{cyril}
\bb E _x \Big\{ e^{T\lambda
 \langle \mu_T, r  \rangle} \Big\}\leq c\, e^{cT }\,, \qquad \forall T\geq 0\,.
 \end{equation}
 Indeed, this follows from \cite{BFG}[Lemma 3.5] (if instead of condition $C(\sigma)$ one assumes Items (i) and (ii) of
  the hypercontractivity Condition  2.4 in \cite{BFG}, then \eqref{cyril} follows from \cite{BFG}[Eq. (3.12)]).
  
 Fixed $\lambda $ as above, we introduce the set
 $ \tilde E:=\{ (y,z) \in \hat E \,:\, H(y) \leq \lambda\}$ and then   define the function $F: E \to [0, +\infty) $ as
$$ F(y,z):= \begin{cases}
\log \frac{\lambda}{ H(y)} & \text{ if } (y,z) \in \tilde E\,,\\
0 & \text{ if } (y,z) \in E \setminus \tilde E\,.\end{cases}
$$
 Defining $r^F(y,z):= r(y,z) e^{F(y,z)}$  we get (recall that $r(y)= \sum _{z:(y,z)\in E} r(y,z)$):
\begin{equation}
\begin{split}
r^F(y):& = \sum _{z:(y,z)\in E} r^F(y,z) = \sum _{z: (y,z) \in E \setminus \tilde E} r(y,z) +
\frac{\lambda}{H(y)}   \sum _{z: (y,z) \in \tilde  E} r(y,z)\\ & \leq
\sum _{z: (y,z) \in E } r(y,z) +
\frac{\lambda}{H(y)}   \sum _{z: (y,z) \in \hat   E} r(y,z) \leq  (1+\lambda) r(y)
\,.
\end{split}
\end{equation}
In particular, we conclude that $r^F(y)-r (y) \leq \lambda r(y)$ for all $y \in V$.
Since $r^F(y) < +\infty$ for all $y \in V$,  by Lemma 3.1 in \cite{BFG} we get  that
\begin{equation}\label{marsiglia}
\bb E _x \Big\{ e^{T\bigl [\langle Q_T,F \rangle - \lambda \langle \mu_T, r  \rangle\bigr ]} \Big\}\leq \bb E _x \Big\{ e^{T\bigl [\langle Q_T,F \rangle - \langle \mu_T, r^F-r \rangle\bigr ]}\Big\} \leq 1\,.
\end{equation}
By Schwarz inequality, combining \eqref{cyril} and  \eqref{marsiglia}, we get for some $C>0$:
\begin{equation}\label{concludo}
 \bb E _x \Big\{ e^{\frac{T}{2} \langle Q_T,F \rangle } \Big\}\leq
 \bb E _x \Big\{ e^{T\bigl [\langle Q_T,F \rangle - \lambda \langle \mu_T, r  \rangle\bigr ]}\Big\}^{\frac{1}{2}}
 \bb E _x \Big\{ e^{T\lambda
 \langle \mu_T, r  \rangle\bigr ]} \Big\}^\frac{1}{2}
  \leq C  e^{C T }\,.
 \end{equation}

Take $a < a_0$ and recall the properties of $W(x,a)\subset E$  given in Item (iii). Since $x$ is fixed, we write simply $W(a)$.  By  assumption $W(a)^c$ is a  finite set.  Given an integer $n\geq 1$ let  $a_n :=\lambda/e^{n^2}$. In particular if $H(y) \leq a_n$ it must be $H(y) < \lambda$ and
$  \ln (\lambda /H(y) )\geq n^2$. We conclude that
\begin{equation}\label{pasqua}
F (y,z) \geq n^2 \qquad \forall (y,z) \in \tilde E \cap W(a_n) = \hat E\cap W(a_n)\,.
\end{equation}
Since $F$
is a nonnegative function, combining \eqref{concludo} with \eqref{pasqua} we get
\begin{equation}\label{concludo_biz}
 \bb E _x \left\{ e^{\frac{n^2 }{2} TQ_T( \hat E \cap W(a_n))   } \right\}\leq C  e^{C T }\,.
 \end{equation}
On the other hand, by  applying Item (iii)--(3) to the family of consecutive states visited by the trajectory $(X_t)_{t \in [0,T]}$,  we get that  $TQ_T( \hat E \cap W (a_n))  \geq \gamma TQ_T
\bigl( W (a_n)  \bigr)$. Hence we conclude that
\begin{equation}\label{concludo_tris}
 \bb E _x \Big\{ e^{\frac{n^2 \gamma}{2 } TQ_T(  W (a_n))   } \Big\}\leq C  e^{C T }\,.
 \end{equation}
Consider now the set  $E_n := W(a_n)^c$, which is finite by Item (iii)--(1). By Chebyshev inequality and \eqref{concludo_tris} we obtain
\[
\bb P_x \left( Q_T(E_n^c) \geq \frac{1}{n}\right)= \bb P_x \left( \frac{n^2 \gamma}{2} T Q_T( W(a_n) ) \geq \frac{n\gamma T}{2} \right)\leq C e^{- \frac{n\gamma T}{2} +CT}\,,
\]
thus leading to \eqref{sandra}.
\end{proof}

\section{Joint large deviations for the empirical measure and current}
\label{s:ec}

Recalling that $E$ denotes the set of ordered edges in $V$ with
strictly positive jump rate, we let $E_\mathrm{s}:= \big\{(y,z)\in
V\times V :\, (y,z)\in E \textrm{ or } (z,y)\in E\big\}$ be the
symmetrization of $E$ in $V\times V$. We then introduce $
L^1_\mathrm{a} (E_\mathrm{s})$ as the space of antisymmetric and
absolutely summable functions on $E_\mathrm{s}$, i.e.\
\begin{equation*}
  L^1_\mathrm{a} (E_\mathrm{s}) :=
  \big\{ J\in L^1 (E_\mathrm{s})\,:\: J(y,z) = - J(z,y)
  \;\; \forall \, (y,z) \in E_\mathrm{s} \big\}\,.
\end{equation*}
Elements of $L^1_\mathrm{a}(E_\mathrm{s})$ will be denoted by $J$
and called \emph{currents}. We shall consider
$L^1_\mathrm{a}(E_\mathrm{s})$ endowed either with  the bounded weak* topology or with the strong topology, and
the associated Borel $\sigma$--algebra.

To each flow $Q\in L^1_+(E)$, we associate the \emph{canonical current}
$J_Q$ defined by
\begin{equation}
  \label{currente_flusso}
  J_Q(y,z):=
  \begin{cases}
    Q(y,z) -Q(z,y) &\textrm{if $(y,z)\in E$ and $(z,y)\in E$,}
    \\
    Q(y,z) &\textrm{if $(y,z)\in E$ and $(z,y)\not\in E$,}
    \\
    -Q(z,y) &\textrm{if $(y,z)\not\in E$ and $(z,y)\in E$.}
  \end{cases}
\end{equation}
Given a current $J$ we
define its divergence, $\div J \in L^1(V)$ by
\begin{equation*}
  \div J(y):=\sum_{z\,:\, (y,z)\in E_\mathrm{s}}J(y,z).
\end{equation*}
It is simple to check the above definition is consistent with
\eqref{divergenza_fluss} in the sense that $\div J_Q= \div Q$.

Given  $T>0$, the \emph{empirical current} is the map
$J_T\colon D(\bb R_+;V)\to L^1_\mathrm{a}(E_\mathrm{s})$ defined as \begin{equation}
  \label{montecarloc}
  \begin{split}
  J_T(y,z) \, (X) :&=
  \frac{1}{T} \sum_{0\leq t\leq  T}
  \big[\mathds{1}\bigl(X_{t^-}=y\,,\,X_{t}=z\bigr) - \mathds{1}\bigl(X_{t^-}=z\,, X_{t}=y\bigr)\big]\\
  & =Q_T(y,z)(X)- Q_T(z,y) (X)   \end{split}
\end{equation}
for all $(y,z ) \in E_s$, where  $Q_T(a,b):=0$ if $(a,b)\in E_s \setminus E$.
Namely, $T\,J_T(y,z)$ is   the  \emph{net} number of
jumps across $(y,z)\in E_\mathrm{s}$ in the time interval $[0,T]$. Equivalently, the empirical current $J_T$
is the canonical current associated to the empirical flow $Q_T$, i.e.\ $J_T= J_{Q_T}$.

Recalling \eqref{Qmu}, to each probability $\mu\in\mc P(V)$ we
associate the current $J^\mu := J_{Q^\mu}$, i.e.\ $J^\mu(y,z) =
\mu(y) r(y,z) -\mu(z) r(z,y)$. Observe that $J^\mu$ has vanishing
divergence if and only if $\mu=\pi$ and $J^\mu$ vanishes if and only
if the chain is reversible with respect to $\mu$. In view of the
discussion in Subsection~\ref{s:emef}, for each $x\in V$ and
$(y,z)\in E_\mathrm{s}$ the sequence of real random variables
$\{J_T(y,z)\}$ converges, in probability with respect to $\bb P_x$,
to $J^\pi(y,z)$ as $T\to+\infty$.

To state the joint  LDP for $(\mu_T, J_T)$ we introduce the function
$\Psi: \bb R\times \bb R\times \bb R_+ \mapsto [0, +\infty)$ given by
\begin{equation}\label{pipistrello}
\Psi ( u, \bar u; a):=
\begin{cases} 
u \left[ {\rm arcsinh\,} \frac{u}{a} - {\rm arcsinh\,} \frac{\bar u }{a} \right] - \left[ \sqrt{a^2+u^2} -
\sqrt{a^2 + \bar{u}^2 }\right] & \text{ if } a>0\,,\\
\Phi(u, \bar u)  & \text{ if } a=0\,.
\end{cases}
\end{equation}

 Due to  the continuity of the map $Q\mapsto J_Q$ the joint large
deviation principle for the empirical measure and current follows
from Theorem~\ref{LDP:misura+flusso} by contraction:

\begin{theorem}
  \label{LDP:misura+cor}
  Assume the Markov chain satisfies (A1),(A2), (A3) and  Condition $C(\sigma)$ with $\sigma >0$.
    Then, as
  $T\to+\infty$, the sequence of probability measures $\{\bb P_x\circ
  (\mu_T,J_T)^{-1}\}$ on $\mc P(V)\times L^1_\mathrm{a}(E_\mathrm{s})$
  satisfies a large deviation principle  with good and convex rate function $\tilde I \colon
  \mc P(V)\times L^1_\mathrm{a}(E_\mathrm{s})\to [0,+\infty]$.

   To have  $\tilde I(\mu, J)<+\infty$ it is necessary that
   $\div J=0$, $J(y,z)\geq 0$ for any
  $(y,z)\in E$ such that $(z,y)\not\in E$ and $\langle \mu,r \rangle <+\infty$. When all  these conditions
  are satisfied we have
  \begin{equation}
    \label{rff}
    \tilde I (\mu,J) =I(\mu,  Q^{J,\mu})=
    \displaystyle{
    \sum_{(y,z)\in E} \Phi \big( Q^{J,\mu}(y,z),Q^\mu(y,z) \big)
    }\,,
  \end{equation}
  where   \begin{equation*}
    Q^{J,\mu}(y,z)  := \frac{J(y,z)+\sqrt{J^2(y,z)+4\mu(y)\mu(z)r(y,z)r(z,y)}}{2}\,.
  \end{equation*}
  The above identity \eqref{rff} can also be rewritten as
  \begin{equation}\label{rff_bis}
   \tilde I (\mu,J) =\frac{1}{2} \sum _{(y,z) \in E_s} \Psi \bigl( J(y,z), J^\mu(y,z); a^\mu (y,z) \bigr)
  \end{equation}
where
\[ a^\mu(y,z) := 2 \sqrt{ \mu(y) \mu(z) r(y,z) r (z,y) }
\,.
\]

 Moreover, if
   the conditions of Theorem \ref{freedom} are satisfied, then the above result remains true with $L^1_\mathrm{a}(E_\mathrm{s})$ endowed with the strong $L^1$--topology.

 \end{theorem}
 Note that if $J\in L^1_{\mathrm{a}}(E_{\mathrm{s}})$ and
$\langle \mu,r\rangle<+\infty$ then $Q^{J,\mu}\in L^1_+(E)$, moreover if
$\div J=0$ then also $\div Q^{J,\mu}=0$.  Note also that
$\tilde{I} (\mu,J) = 0$ if and only if $I(\mu,  Q^{J,\mu})=0$, and we know this holds if and only if $\mu= \pi$ and $Q^{J,\mu}= Q^\pi$ (see the discussion after Theorem \ref{LDP:misura+flusso}). It is trivial to check that this last condition is equivalent to
 $(\mu,J) =(\pi,J^\pi)$.

\begin{proof}[Proof of Theorem~\ref{LDP:misura+cor}] Recalling that $L^1_+(E)$ is  equipped with the bounded weak*
topology, the map $L^1_+(E)\ni Q\mapsto J_Q\in
L^1_\mathrm{a}(E_\mathrm{s})$ is continuous. This map remains continuous if $L^1_+(E)$ and
$L^1_\mathrm{a}(E_\mathrm{s})$ are both endowed of the strong $L^1$--topology. Hence,
  by Theorem~\ref{LDP:misura+flusso}, Theorem \ref{freedom} and the  contraction principle, a joint LDP holds for $(\mu_T,J_T)$ with
 good rate function
   \begin{equation*}
    \tilde I (\mu,J) := \inf\big\{ I(\mu,Q)\,:\, Q \in L^1_+(E) \text{ with } J_Q=J\big\}\,, \quad (\mu, J) \in \mc P (V) \times L^1_a (E_s) \,.
  \end{equation*}
 It remains to show that the above $\tilde I(\mu,J)$ fulfills   the properties stated in the theorem.

 From the above variational characterization of $\tilde I (\mu,J)$ one easily derives that $\tilde I$ is convex,
 since $I(\mu,Q)$ is convex and the map $L^1_+(E)\ni Q \to J_Q \in L^1_a (E_s) $ is linear.

  It is simple to verify that   to have $\tilde I(\mu, J)<+\infty$ it is necessary that
   $\div J=0$, $J(y,z)\geq 0$ for any
  $(y,z)\in E$ such that $(z,y)\not\in E$ and $\langle \mu,r \rangle <+\infty$. Let us now take $(\mu,J) \in \mc P(V)\times L^1_\mathrm{a}(E_\mathrm{s})$ satisfying the above three conditions. Then
 the set  $\left\{Q\in L^1_+(E)\, :\, J_Q=J\right\}$ coincides
with the set of flows of the type
$$
Q(y,z)=\begin{cases}
[J(y,z)]_++s(\left\{x,y\right\})\,, & \mathrm{if } \;\;(z,y)\in E\,,\\
[J(y,z)]_+= J(y,z)\,, & \mathrm{if } \;\;(z,y)\not \in E\,,
\end{cases}
$$
where $[\cdot]_+$ denotes the positive part (i.e. $[z]_+:= \max\{0, z\} $), $s\in L^1_+(E_{\mathrm{u}})$
and $E_{\mathrm{u}}:=\left\{\left\{y,z\right\}\,:\, (y,z)\in E_{\mathrm{s}}\right\}$ is
the set of \emph{unordered edges}. We can then solve independently a variational problem for each pair of
edges $(y,z)$ and $(z,y)$ in $E_{\mathrm{s}}$.
If $(y,z)$ and $(z,y)$ both belong to $E$, then
an elementary computation gives  that
\begin{eqnarray*}
& &\inf_{s\in [0,+\infty)}\left\{ \Phi\Big([J(y,z)]_++s,Q^\mu(y,z)\Big)+\Phi\Big([-J(y,z)]_++s,Q^\mu(z,y)\Big)\right\}\\
& & \qquad =
\Phi\Big(Q^{J,\mu}(y,z),Q^\mu(y,z)\Big)+\Phi\Big(Q^{J,\mu}(z,y),Q^\mu(z,y)\Big)\,.
\end{eqnarray*}
If $(y,z) \in E$ and $(z,y) \not \in E$, then $Q(y,z)= J(y,z)=Q^{J,\mu}(y,z)$ for any $Q \in L^1_+(E)$ with $J_Q=J$. Since $\div Q= \div J_Q=0$, by the expression \eqref{rfq} of the rate function $I$ and the above computations, we obtain  that $\tilde I (\mu,Q)= I(\mu, Q^{J,\mu)})$, hence \eqref{rff}.

It remains to prove \eqref{rff_bis}. To this aim we observe that, if both $(y,z)$ and $(z,y)$ belong to $E$, then  the following identities hold:
  \begin{align}
 \Phi \big( Q^{J,\mu}(y,z),Q^\mu(y,z) \big)& + \Phi \big( Q^{J,\mu}(z,y),Q^\mu(z,y) \big)  \nonumber \\
& =  \frac{j }{2} \log \Big[ \frac{ j+ \sqrt{ j^2+4 p p'} }{ -j+ \sqrt{ j^2+4 p p'}} \frac{p'}{p}\Big]- \sqrt{j^2+4 p p'}+p+p' \nonumber \\
 &=  j \log\frac{ j+ \sqrt{ j^2+4 p p'} }{2p}- \sqrt{j^2+4 p p'}+p+p'\label{carmine}
 \end{align}
 where  $j:= J(y,z)$, $p= \mu(y) r(y,z)$, $p' = \mu(z) r(z,y)$, assuming $p,p'$ positive. Set $a:= a^\mu(y,z)= 2 \sqrt{p p'}$ and $\bar j:=J^\mu(y,z)=p-p'$.
Since ${\rm arcsinh\,} u = \log[ u+ \sqrt{u^2+1}  ]$, $j^2+4pp'= j^2+a^2$,
$p+p'=\sqrt{ \bar j^2+a^2 }$, the last member in \eqref{carmine} can be rewritten as
 $\Psi(j, \bar j; a)$.

 Suppose now that $(y,z) \in E$ and $(y,z) \not \in E$, $J(y,z) \geq 0$. Then $Q^{J,\mu}(y,z) = J(y,z)$  and $Q^\mu(y,z)= J^\mu (y,z) $. In particular,
$$ \Phi \big( Q^{J,\mu} (y,z),Q^\mu(y,z) \big)= \Psi( J(y,z), J^\mu(y,z) ; 0)\,.
$$
From the above considerations it is simple to derive \eqref{rff_bis} from \eqref{rff}. 
\end{proof}


\section{Gallavotti--Cohen type symmetries for the empirical current}

In this section and in Sections \ref{bonbon} and \ref{sec:hom}, we
assume that $E_\mathrm{s}= E$ (i.e.  $r(y,z)>0$ if and only if
$r(z,y)>0$) and we derive Gallavotti--Cohen (GC) symmetries of the LD
rate function both of the empirical current and of suitable linear
functionals of the empirical current itself.
  
   In what follows,
  $w_\pi \colon E \to \bb R$  denotes  the  antisymmetric
function 
  \begin{equation}
  \label{wpi}
  w_\pi (y,z) = \log \frac{\pi(y) r(y,z)}{\pi(z)r(z,y)},
  \qquad (y,z)\in E\,,
\end{equation}
  and we will assume that $w_\pi \in L^\infty(E)$, thus implying that $\langle J, w_\pi\rangle$ is finite for any $J \in L^1_\mathrm{a}(E)$.

  \begin{theorem}\label{iacopo_ale}
  Assume $E_\mathrm{s}= E$ and that $ w _\pi \in L^\infty (E)$.
 Then
the rate function $\tilde I$ of  Theorem \ref{LDP:misura+cor} satisfies the following GC  symmetry in $[0,+\infty]$:
\begin{equation}\label{silente2} \tilde I (\mu,J)= \tilde I(\mu,-J)- \frac{1}{2} \langle J, w_\pi \rangle \,,\qquad \forall (\mu,J) \in\mc P(V)\times L^1_\mathrm{a}(E)\,.
\end{equation}
 In particular, the good and convex rate function $\hat I: L^1_a (E) \to [0,+\infty]$, $\hat I(J)= \inf _\mu \tilde I (\mu, J)$, of the LDP for the empirical current obtained  by contraction from Theorem \ref{LDP:misura+cor} satisfies the following  GC symmetry  in $[0,+\infty]$:
\begin{equation}\label{maghetto}
\hat I ( J)= \hat I(-J)- \frac{1}{2} \langle J, w_\pi \rangle \,, \qquad \forall J \in  L^1_\mathrm{a}(E)\,.
\end{equation}
\end{theorem}
\begin{remark} For finite state spaces the  GC symmetry \eqref{maghetto} has already been derived in \cite{AG1,AG2,FD} in terms of the moment generating functions (essentially, by means of G\"artner--Ellis theorem).
\end{remark}

\begin{proof} Having \eqref{silente2}, the conclusion is a trivial consequence of the contraction principle. Let us prove \eqref{silente2}.
If $\div J\not =0$ or $ \langle \mu,r \rangle =+\infty$, then
$\tilde I (\mu,J)= \tilde I(\mu,-J)
=+\infty$ and  \eqref{iacopo_ale} is trivially true (recall that  $\langle J, w_\pi \rangle$ is finite). Suppose therefore that $\div J =0$ and $ \langle \mu,r \rangle <+\infty$. Then, $\tilde I (\mu,J)$ and $\tilde I(\mu,-J)$ have the series expression induced by  \eqref{rff}.
It is simple to check that, given $p , p' >0$ and $q,q' \geq 0$, it holds
\begin{equation}
\Phi(q,p)+ \Phi(q',p')=\Phi(q',p)+\Phi(q,p')+ (q-q') \log (p'/p)\,.
\end{equation} Taking  $q:=Q^{J,\mu} (y,z)= Q^{-J,\mu}(z,y)$, $q':=Q^{J,\mu} (z,y)= Q^{-J,\mu}(y,z)$,
$p:=Q^\mu(y,z)$, $p':=Q^\mu(z,y)$, from the above identity we get
\begin{multline*}\Phi\Big(Q^{J,\mu}(y,z),Q^\mu(y,z)\Big)+\Phi\Big(Q^{J,\mu}(z,y),Q^\mu(z,y)\Big)\\=
\Phi\Big(Q^{-J,\mu}(y,z),Q^\mu(y,z)\Big)+\Phi\Big(Q^{-J,\mu}(z,y),Q^\mu(z,y)\Big)- J(y,z) w_\pi (y,z)\,.\end{multline*}
Summing above $(y,z) \in E$ we get  \eqref{silente2}.\end{proof}

\section{Gallavotti--Cohen  symmetry for the Gallavotti-Cohen functional}\label{bonbon}

Let $\bb P_\pi$
be the law of the stationary chain (the initial state is
sampled according to the invariant probability $\pi$). By
stationarity, $\bb P_\pi$ can be extended to a measure on $D(\bb
R;V)$. Let $\vartheta\colon D(\bb R;V) \to D(\bb R;V)$ be the time
reversal, i.e.\ for the set of times $t\in \bb R$ which are
continuity points of $X$ the map $\vartheta$ is defined by
$(\vartheta X)_t =X_{-t}$. We then set $\bb P_\pi^* := \bb P_\pi
\circ \vartheta^{-1}$; of course $\bb P_\pi^*= \bb P_\pi$ if and
only if the chain is reversible. In general, $\bb P_\pi^*$ is the
law of the stationary chain with jump rates $r^*(y,z) = \pi(z)
r(z,y) / \pi(y)$.  Given $x\in V$ and $T>0$, the Gallavotti-Cohen
functional can be defined (cf. \cite{LS}) as the map $W_T\colon D(\bb R_+;V)\to \bb
R$ which is $\bb P_x$ a.s.\ given  by
\begin{equation}
  \label{gc}
  W_T := - \frac 1T \log
  \frac{d \bb P_\pi^* \big|_{[0,T]}}{d \bb P_\pi\big|_{[0,T]}}.
\end{equation}
Observe that $\bb E_\pi\big( W_T \big)$ is $(1/T)$--proportional to the
relative entropy of $\bb P_\pi\big|_{[0,T]}$ with respect to $\bb
P_\pi^*\big|_{[0,T]}$, thus providing a natural measure of the
irreversibility of the chain.

A simple  computation of the  Radon-Nikodym derivative in \eqref{gc}
(use  (3.1) in \cite{BFG} and observe that $r^*(y)=r(y)$ for any $y \in V$ due to the invariance of $\pi$) gives that
 the Gallavotti-Cohen functional $W_T$ can be written in
terms of the empirical current $J_T$ as
\begin{equation}
  \label{gcec}
  W_T  =  \frac 12 \, \langle J_T, w_\pi \rangle\,,
\end{equation}
where  $w_\pi \colon E\to \bb R$ is the  antisymmetric
function defined by \eqref{wpi}.

As a consequence of our previous results and the contraction principle we get the following LDP:

\begin{theorem}
  \label{LDP:GC}
   Assume  that $E=E_s$, the Markov chain satisfies (A1),(A2), (A3) and assume   Condition $C(\sigma)$ with $\sigma >0$.
      Assume also that $w_\pi$    vanishes  at infinity. If the conditions of Theorem \ref{freedom} are satisfied, it is enough to require that $w_\pi$ is a bounded function.

  Then, as $T\to+\infty$, the sequence of probability measures
  $\{\bb P_x\circ W_T^{-1}\}$ on $\bb R$ satisfies a large
  deviation principle
  with good and convex rate function $\imath \colon \bb R\to
  [0,+\infty]$ given by
  \begin{equation}
    \label{vcgc}
    \imath(u) = \inf\big\{ \tilde{I}(\mu,J)\,:\:  (\mu,J) \in \mc P(V)\times
  L^1_\mathrm{a}(E)\,,\;
    \langle J,w_\pi\rangle = 2 u \big\}\,.
  \end{equation}
  Moreover, the following GC  symmetry  holds in $[0,+\infty]$:
  \begin{equation}\label{2010}
    \imath (u) =\imath (-u)-u\,.
  \end{equation}
\end{theorem}
\begin{proof}
Note that the map $L^1_\mathrm{a}(E ) \ni J \to \langle J, w_\pi \rangle \in \bb R$ is well defined and continuous in both the following cases: (i) $L^1_\mathrm{a}(E )$  is endowed of the bounded weak* topology and $w_\pi $ vanishes at infinity, (ii) $L^1_\mathrm{a}(E )$  is endowed of the strong $L^1$--topology and $w_\pi$ is bounded. Hence, due to the contraction principle and  Theorem  \ref{LDP:misura+cor},  we only need to prove  that the rate function $\imath ( \cdot)$ is convex
and that GC  symmetry \eqref{2010} is fulfilled.  The last property follows from Theorem \ref{iacopo_ale}.
The convexity follows easily from the fact that $\tilde I (\mu,J)$ is convex and the constraint $\langle J,w_\pi\rangle = 2 u$ is linear in $J$.
\end{proof}

The Gallavotti-Cohen functional is  defined in \cite{LS} by replacing the function $w_\pi$ above with $w(y,z) =
\log[ r(y,z) / r(z,y)]$.
 In order to be able to discuss applications
to Markov chains with infinitely many states we have chosen the
previous definition with $w_\pi$.  Note that
\begin{equation}
\frac{1}{2}\langle J_T, w _\pi \rangle - \frac{1}{2}\langle J_T,w \rangle = \frac{1}{2} \sum_{(y,z)\in E} J_T(y,z) \log \frac{\pi(y)}{\pi(z)} = \frac{1}{T} \log \frac{\pi(X_T) }{\pi(X_0)}\,.
\end{equation}
Hence,  if $V$ is finite, the term $ \log \frac{\pi(X_T) }{\pi(X_0)}$ is bounded, thus implying that
$\frac{1}{2}\langle J_T, w _\pi \rangle $ and $ \frac{1}{2}\langle J_T,w \rangle $ satisfy the same LDP.
Theorem \ref{LDP:GC}
 provides a variational
characterization of the rate function for the Gallavotti-Cohen
functional which can be compared to the rather implicit one derived
e.g.\ in \cite[with $w$ instead of $w_\pi$]{LS} by using the Perron-Frobenius and the
G\"artner-Ellis theorems.

\section{LDP for the homological coefficients and Gallavotti--Cohen symmetry}\label{sec:hom}

Also in this section we assume that the graph $G=(E,V)$ has the property
$(y,z) \in E \; \Leftrightarrow \; (z,y)\in E$ (i.e. $E=E_s$) and
extend to the infinite case the concept of \emph{cycle space}. We
refer e.g.  to \cite{AG1,AG2,FD,S} for physical applications and
e.g. to \cite{Bo,D} for a mathematical treatment in finite graphs. We
also prove that the cycle space is isomorphic to the first cellular
homological class over $\bb R $ of the graph $G$ (shortly, $H_1(G, \bb
R) $). Then we associate to each trajectory up to time $T$ a cycle
$\mc C_T$ and prove a LDP for the empirical homological coefficients,
which are given by the coefficients in a given basis of the cycle $\mc
C_T$ thought of as element of the cycle space, and therefore of
$H_1(G, \bb R) $.

\subsection{Cycle space of the graph $G$} We point out that, working
with a graph $G=(V,E)$ with $E=E_s$, all information encoded in $G$
corresponds to the one encoded in its unoriented version
$G_u=(V,E_u)$, where $E_u := \bigl\{ \, \{y,z\}\,:\, (y,z) \in
E\bigr\}$.  The subscript "u" stays for \emph{unoriented}.  Hence, the
discussion that follows applies as well to unoriented graphs.

We  fix some notation.
Given an edge $e=(y,z)\in E$ we write $\bar e=(z,y)$ for the reversed edge. A  \emph{cycle} $\mc C$ in
$G$  is a finite string
$(x_1,\ldots,x_k)$ of elements of $V$ such that $(x_{i},x_{i+1})\in E$
when $i=1,\dots, k$, with the convention that $x_{k+1}=x_1$. Given a cycle
$\mc C$ and given $e\in E$ we definite $S_e( \mc C)$ as the number of times the edge $e$ appears in $\mc  C$ minus the number of times the reversed  edge $\bar e$ appears in $\mc C$:
$$ S_e (\mc C):=\sharp \{ i \,:\, 1 \leq i \leq k \,,\; (x_i, x_{i+1})=e \}-\sharp \{ i \,:\, 1 \leq i \leq k \,,\; (x_i, x_{i+1})=\bar e \}\,.$$

Consider now the free real vector space $\mc V$ generated by all cycles $\mc C$. Its elements are the formal sums $\sum _{j=1}^n a_j \mc C_j$, varying   $n \in \bb N$,  $a_j \in \bb R$ and $\mc C_j$ cycles, with the natural rules for sum and multiplication by a   constant.   The empty sum is the zero element of $\mc V$, denoted by $\emptyset$.

The \emph{cycle space} $\mc V_*$ of the graph $G$ is then defined as the  quotient vector space of $\mc V$ imposing that  in $\mc V_*$ it holds
\begin{equation}\label{scappiamo}
\sum _{j=1}^n a_j \mc C_j
= \sum _{i=1}^m b_i \mc C'_i \qquad \text{ iff } \qquad  \sum _{j=1}^n a_j S_e(\mc C_j)
= \sum _{i=1}^m b_i S_e(\mc C'_i)\;\; \forall e \in E
\end{equation}
(we keep the same notation for the elements of $\mc V$ and $\mc V_*$). More precisely,
$ \mc V_*$ is defined as the quotient $\mc V / \mc W$, where the subspace $\mc W$ is given by the sums
$\sum _{j=1}^n a_j \mc C_j
-\sum _{i=1}^m b_i \mc C'_i$ satisfying the identity system in the r.h.s. of \eqref{scappiamo}. Note that
 in the cycle space  $\mc V_*$ the cycle $\mc C=( x_1, x_2, \dots, x_k)$ equals the cycle $\mc C'=(x_i,x_{i+1}, \dots, x_k,x_1, \dots , x_{i-1})$, and that $-\mc C = (x_k, x_{k-1}, \dots, x_1)$. 

Special bases (called \emph{fundamental bases})  of $\mc V_*$ can be obtained starting from a spanning tree  $\mc T= (V, E_{\mc T})$ of the  unoriented graph $G_u=(V,E_u)$.
   Fix such a spanning tree $\mc T$. To each edge in  $E_u \setminus E_{\mc T}$  we assign an orientation and we call 
\emph{chords} the resulting oriented edges.\footnote{Usually, chords are the unoriented edges in  $E_u \setminus E_{\mc T}$  \cite{Bo,D,S}.  To avoid additional notation we have directly included in their definition a fixed orientation.}    
     To each chord $\mathfrak{c}$ we associate a cycle $\mc C_{\mathfrak{c}}\in \mc V_*$ as follows:  
   consider  the unique self-avoiding path $x_1,x_2, \dots, x_k$ in $G$ such that  $\mathfrak{c}=(x_1,x_2)$ and $\{x_i,x_{i+1}\} \in E_{\mc T}$ for all $i=2,3, \dots, k$,  and  set 
 $\mc C_{\mathfrak{c}}:= (x_1,x_2, \dots, x_k)$. Note that by construction
 \begin{equation}\label{deltino} S_{\mathfrak{c}   } (\mc C _{\mathfrak{c}'}) = \delta _{\mathfrak{c}, \mathfrak{c}'}\,.
 \end{equation}

 \begin{proposition}
Given a spanning tree $\mc T$ of the unoriented graph $G_u=(V,E_u)$, the cycles  $\mc C_{\mathfrak{c}}$ -  with $\mathfrak{c}$ varying among the chords of $\mc T$ - form a basis of the quotient space $\mc V_*$. Moreover, for each cycle $\mc C$ the following identity holds  in $\mc V_*$:
\begin{equation}\label{lampone}
\mc C= \sum _{\mathfrak{c} } S_{ \overrightarrow{\mathfrak{c}}   } \bigl( \mc C \bigr)  \mc C _{\mathfrak{c} }\,.
\end{equation}
 \end{proposition}
 We call the above basis   $\{\mc C_{\mathfrak{c}}\,: \, \mathfrak{c} \text{ chord of } \mc T \}$ a fundamental basis associated to the spanning tree $\mc T$ (see Figure \ref{codroipo1}). Not all basis of $\mc V_*$ are fundamental, as can be seen  e.g. from the simple example given in \cite{FD}[Section 7]. Due to the above proposition and \eqref{deltino}, the linear functions $ \mc V_* \ni \sum_{i=1}^n a_i \mc C_i \mapsto \sum_{i=1}^n a_i S_{\mathfrak{c}  } ( \mc C_i )\in \bb R$, as $\mathfrak{c}$ varies among the chords,  form the dual basis of $\{\mc C_{\mathfrak{c}}\,: \, \mathfrak{c} \text{ chord of } \mc T \}$.

\begin{figure}[htbp!]
\centering
\bigskip
\bigskip
\begin{overpic}[scale=0.5]{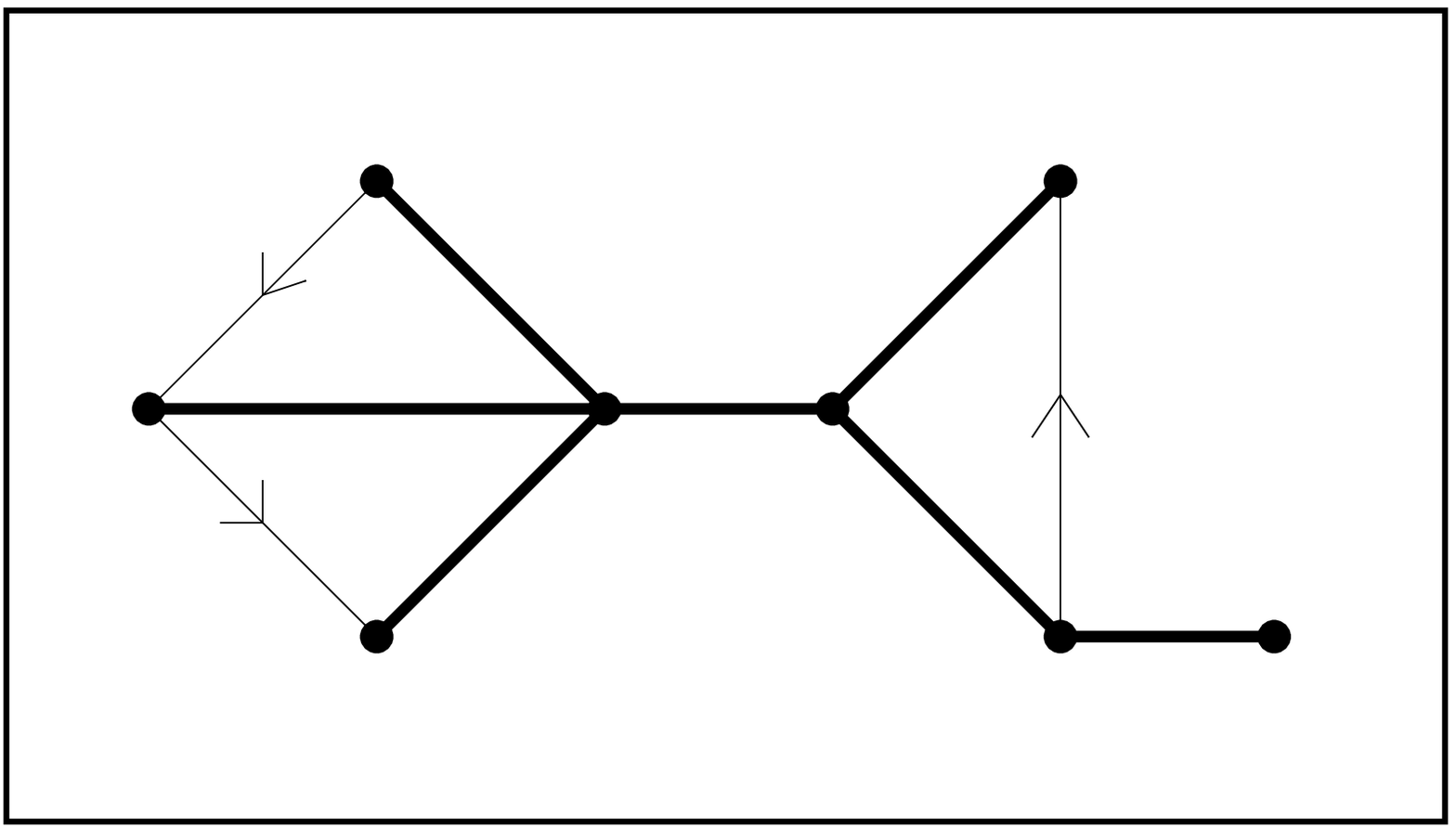}
\put(23,48) {$x_1$}
\put(6,23) {$x_2$}
\put(23,8) {$x_3$}
\put(41,23) {$x_4$}
\put(53,23) {$x_5$}
\put(72,8) {$x_6$}
\put(72,48) {$x_7$}
\put(86,8) {$x_8$}
\put(14,40){$\mathfrak{c}_1$}
\put(14,16){$\mathfrak{c}_2$}
\put(75,30) {$\mathfrak{c}_3$}
\end{overpic}
\bigskip
\caption{Fundamental basis. The bold edges form the spanning tree $\mc T$. The chords $\mathfrak{c}_1=(x_1,x_2)$,  $\mathfrak{c}_2=(x_2,x_3)$ and   $\mathfrak{c}_3=(x_6,x_7)$ correspond to the cycles $\mc C_{\mathfrak{c}_1}=(x_1,x_2,x_4)$, $ \mc C_{\mathfrak{c}_2}=(x_2,x_3,x_4)$ and $\mc C_{\mathfrak{c}_3}= (x_6,x_7,x_5)$, respectively.  }
\label{codroipo1}
\end{figure}

The above proposition is a classical result in the finite setting (cf. \cite{D} when working with the field $\bb F_2$ instead of $\bb R$). The proof for infinite graphs could be recovered by the result for finite graphs. For completeness we give a direct and self--contained proof. 

 \begin{proof} 
   If $\mc C= \sum _{i=1}^n a_i \mc C_{\mathfrak{c}_i }$ with chords
   $\mathfrak{c}_1, \mathfrak{c}_2, \dots, \mathfrak{c}_n$ all
   distinct, by applying $ S_{ {\mathfrak{c}_j} } $ and invoking
   \eqref{deltino} we get that $a_j= S_{ {\mathfrak{c}_j} } (\mc C)$.
   This proves \eqref{lampone} for any cycle $\mc C$ that is generated
   by the \emph{fundamental cycles} $\mc C_\mathfrak{c}$'s.
   We thus need to prove that these cycles form a basis.

 We first prove that the cycles $\mc C_{\mathfrak{c}}$'s are linearly
 independent.  Suppose that $\sum_{i=1}^n a_i \mc
 C_{\mathfrak{c_i}}=0$ for some constants $a_1, \dots, a_n$ and some
 chords $\mathfrak{c_1}, \dots, \mathfrak{c_n}$.  By \eqref{scappiamo}
 and \eqref{deltino} one easily gets that $a_i=0$ for all $i$, hence
 the independence.

We now prove that the cycles $\mc C_{\mathfrak{c}}$'s generate all
$\mc V_*$. To this end, it is enough to show that they generate any
cycle $\mc C$. Since any cycle $\mc C$ is in $\mc V_*$ the sum of
self--avoiding cycles, we can restrict to a self--avoiding cycle $\mc
C$, i.e. $\mc C=(x_1, x_2, \dots, x_k)$ with $x_1,x_2, \dots, x_k$ all
distinct (recall that it must be $(x_i,x_{i+1})\in E$ for all
$i=1,2,\dots, k$ with the convention $x_{k+1}=x_k$). We prove
that the self--avoiding cycle $\mc C$ can be expressed as linear
combination of the fundamental cycles $\mc C_\mathfrak{c}$'s
by induction on the cardinality of the set
 \begin{equation}\label{cavallino}
 \bigl\{ \mathfrak{c} \text{ chord }:S_{ \mathfrak{c}  }  ( \mc C) \not = 0 \bigr\}\,.
 \end{equation}
 If the above set has zero cardinality, i.e. it is empty, then,  
as $\mc C$ is self--avoiding and $\mc T$ is a tree, then $\mc C =
(x_1,x_2)$, which is indeed zero in $\mc V_*$.
Given a positive integer $m$, let us now suppose that $\mc C$ is
generated by fundamental cycles when the set \eqref{cavallino} has
cardinality less then $m$.  Take a self--avoiding cycle $\mc
C=(x_1,x_2, \dots, x_k)$ such that the set \eqref{cavallino} has
cardinality $m$ and fix a chord $ {\mathfrak{c}}_* $ inside
\eqref{cavallino}. Without restriction we can suppose that $(x_2,
x_1)= \mathfrak{c}_* $ (at cost to replace $\mc C$ by $- \mc C$ and to
relabel the points $x_1, x_2, \dots, x_n$). The cycle $\mc
C_{{\mathfrak{c}}_*}$ is then of the form $(x_2, x_1, y_3, \dots ,
y_r)$, where $\{x_1, y_3\}$, $\{y_3,y_4\}$,...,$\{y_{r-1}, y_r\}$, $\{
y_r, x_2\}$ are edges of the tree $T$. Consider now the cycle
 \[ \bar{ \mc C}:= (x_1, y_3, \dots , y_r, x_2, x_3, \dots, x_m)\,, \]
obtained by removing from $\mc C$ the edge $(x_1,x_2) $ and replacing
it with the path $ x_1, y_3, \dots , y_r$. Note that $ \mc C= \bar{\mc
  C} -\mc C_{{\mathfrak{c}}_*}$ in $\mc V_*$. By construction, 
\begin{equation}\label{bionda1}
\bigl\{ \mathfrak{c} \text{ chord }: S_{  {\mathfrak{c}}  } ( \bar{\mc C}) \not = 0 \bigr\}=
 \bigl\{ \mathfrak{c} \text{ chord }: S_{ {\mathfrak{c}}  } ( \mc C) \not = 0 \bigr\}\setminus \{\mathfrak{c}_* \}\,.
 \end{equation}
 At this point, write $\bar{\mc C}$ as sum  $\sum _{u=1}^s \bar {\mc C}_u$ of self--avoiding cycles simply by cutting $\bar {\mc C}$ at its intersection points. Since the support of $\bar{\mc C}_u$ is included in the support of $\bar {\mc C} $ we have
 \begin{equation}\label{bionda2}
 \bigl\{ \mathfrak{c} \text{ chord }: S_{ {\mathfrak{c}}  }  ( \bar{\mc C}_u) \not = 0 \bigr\}\subset
\bigl\{ \mathfrak{c} \text{ chord }: S_{  {\mathfrak{c}}  } ( \bar{\mc C}) \not = 0 \bigr\} \,,
 \end{equation}
 hence by \eqref{bionda1} the set in the l.h.s. of \eqref{bionda2} has cardinality less than $m$. By applying the inductive hypothesis we finally get
 \[ \bar{\mc C}_u= \sum _{\mathfrak{c} } S_{ {\mathfrak{c}}  }\bigl( \bar{\mc C}_u \bigr)  \mc C _{\mathfrak{c} }\,.\]
 Putting all together we then conclude
 \[
 \mc C = 
 \bar{\mc C}- \mc C_{{\mathfrak{c}}_*}
= \sum _{u=1}^s \bar {\mc C}_u - \mc C_{ {\mathfrak{c}}_*}
= \sum _{u=1}^s\sum _{\mathfrak{c} }
 S_{\mathfrak{c} }\bigl( \bar{\mc C}_u \bigr) \mc C _{\mathfrak{c}
 }- \mc C_{{\mathfrak{c}}_*}\,,
 \]
 hence $\mc C$ is a (finite) linear combination of cycles $\mc
 C_{\mathfrak{c}}$'s. By applying \eqref{deltino} one gets that
 \eqref{lampone} is satisfied. 
 \end{proof}


 \subsection{Cellular homology}
 Consider the graph $G_u= (V, E_u)$, for each unordered edge in $E_u$
 fix a canonical orientation and call $E_o$ the set of canonically
 ordered edges (the subscript ``o" stays for \emph{ordered}, or
 \emph{oriented}). In other words, $E_o$ is any subset $E_o \subset E$
 such that if $(y,z) \in E$ then either $(y,z) \in E_o$ or $(z,y)\in
 E_o$.

We recall the definition of the first cellular homology class  $H_1( G, \bb R)$ (the field $\bb R$ could be replaced by a generic ring $\bb F$). To this aim, we introduce a proper terminology: the vertexes in $V$ are  called 
 $0$--cells and  the edges in $E_o$ are  called  $1$--cells. For $k=0,1$ we define  the space $C_k(\bb R)$ of $k$--chains  as the free vector space  over $\bb R$  (in general as the free $\bb F$--module) with basis given by the $k$--cells.
  Finally, we define the boundary operator
 \[ \partial : C_1 (\bb R) \mapsto C_0 (\bb R) \]
 as the unique linear map such that $\partial (y,z)= z-y$ for any $(y,z) \in E_o$.
The first cellular homology class $H_1( G, \bb R)$ is then given by the kernel of $\partial$. We point out that the definition depends on the choice of the set $E_o$ of canonically oriented edges, but any other choice of $E_o$ would lead to a isomorphic vector space.

Since    the graph has no facets of dimension $2 $, the family of $2$--cells is empty  and the space $C_2(\bb R)$ of $2$--chains is zero, hence the boundary operator from $C_2(\bb R)$ to $C_1(\bb R)$ would be the zero map. In particular, the zero $1$--chain is the only exact chain, while the closed 1-chains form the kernel of the boundary operator
 $ \partial : C_1 (\bb R) \mapsto C_0 (\bb R) $. Hence the above definition of $H_1( G, \bb R)$ coincides indeed with the standard one, as  quotient of the closed 1-chains over the exact 1-chains.

To a  given a cycle $\mc C=(x_1, x_2, \dots, x_k)$  in $\mc V$  we associate the homological class
\begin{equation}\label{setona}
[\mc C]:=\sum _{e \in E_o} S_e ( \mc C) e
\end{equation} in $H_1( G, \bb R)$.
 Note that  the above series is indeed a finite sum and that $\partial [\mc C]=0$.

Then we have the following result:
\begin{proposition}
The linear map $\psi: \mc V \ni  \sum _{i=1}^n a_i \mc C_i \mapsto \sum _{i=1}^n a_i [\mc C_i] \in H_1( G, \bb R)$ induces  the quotient linear map
\[\phi: \mc V _* \ni  \sum _{i=1}^n a_i \mc C_i \mapsto \sum _{i=1}^n a_i [\mc C_i] \in H_1( G, \bb R)\,,\] which is a linear isomorphism.
\end{proposition}
 \begin{proof}
 To see that the map $\phi$ is well defined, we need to show that $\psi$ is zero on $\mc W$ (recall that $\mc V_* = \mc V / \mc W$). To this aim, given  $\sum_{i=1}^n a_i \mc C_i$  in  $\mc V$  such that $\sum_{i=1}^n a_i S_e( \mc C_i)=0 $ for any $e \in E$, we  have to prove that $\sum_{i=1}^n a_i [\mc C_i]=0$. This follows easily from  definition \eqref{setona}.

 Let us prove that $\phi$ is injective. Suppose that, for some $\sum _{i=1}^n a_i \mc C_i\in \mc V_*$, it holds $ \sum _{i=1}^n a_i [\mc C_i] =   0$. Since,  by \eqref{setona}, $ \sum _{i=1}^n a_i [\mc C_i]  = \sum _{e \in E_o} \bigl( \sum _{i=1}^n a_iS_e ( \mc C_i)\bigr)e $ (note that the series over $e \in E_o$ is   indeed a  finite sum) we conclude that $\sum _{i=1}^n a_iS_e ( \mc C_i)=0$ for any $e \in E_o$, which implies that $\sum _{i=1}^n a_i \mc C_i=0$ in $\mc V_*$ by \eqref{scappiamo}.

  Let us prove that $\phi$ is surjective. To this aim fix $f = \sum _{e \in E_0 }b_e e$ in $H_1( G, \bb R)$ (in particular, the above series over $e\in E_o$ is a finite sum).  Since $ \phi( \emptyset)=0$ we can assume $f \not =0$. We define the flow $Q\in L^1_+(E)$ as follows:
 for any $e \in E_o$ with $b_e >0 $ we put $Q(e):= b_e$, while for any $e \in E_o$ with $b_e < 0$ we put $Q(\bar e):=- b_e$, and we set the flow $Q$ equal to zero in all other edges.  By the above definition it is simple to check that $Q(e)- Q(\bar e)= b_e $ for any $e \in E_o$. We now show that $\div Q=0$. Indeed
\begin{equation*}
\begin{split} \div Q(y)&= \sum _z \bigl( Q(y,z)- Q(z,y)\bigr)\\& = \sum _{z: (y,z) \in E_o}
 \bigl( Q(y,z)- Q(z,y)\bigr) +\sum _{z: (z,y) \in E_o}
 \bigl( Q(y,z)- Q(z,y)\bigr)\\
 &= \sum _{z: (y,z) \in E_o} b_{(y,z)} - \sum _{z: (z,y) \in E_o} b_{(z,y)}\,.
 \end{split}
 \end{equation*}
      On the other hand
      the last member equals the value
      of the $0$--chain $-\partial f$ in $y$  and we know that $\partial f =0$, thus proving the zero--divergence of  $Q$. By Lemma 4.1 in \cite{BFG}  and since the flow $Q$ has finite support, we then conclude that there exist  self--avoiding cycles $\mc C_1, \mc C_2 ,\dots,  \mc C_n$
      and positive constants $a_1,a_2, \dots, a_n$ such that
      \begin{equation}\label{biauzzo}
      Q(e)= \sum _{i=1}^n a_i \mathds{1}(e \in \mc C_i)\,.
      \end{equation} In general  we write $e \in \mc C$ if $\mc C=(x_1, x_2, \dots, x_k)$ and  $e=(x_j,x_{j+1}) $ for some $j\in \{1,2, \dots , k \} $. We claim that $\phi$ maps $\sum _{i =1}^n a_i \mc C_i$, thought of as element of $\mc V_*$, to $f\in H_1( G; \bb R)$. To this aim we need to show that
      \begin{equation}\label{tagliamento}
      \sum _{i=1}^n a_iS_e ( \mc C_i)=b_e
      \end{equation}
      for any
        $e \in E_o$.   If $b_e=0$, then by construction $Q(e)=Q(\bar e)=0$, hence by \eqref{biauzzo} $e, \bar e$ are not in the support of the $\mc C_i$'s, thus implying \eqref{tagliamento}.      If $b_e >0$ then $Q(e)= b_e$ and $Q(\bar e)=0$, hence by \eqref{biauzzo}
     $ \sum _{i=1}^n a_i \mathds{1}(e \in \mc C_i)=b_e$, while
           $ \bar e$ is not in the support of the $\mc C_i$'s. This implies \eqref{tagliamento}. If $b_e <0$, then
$Q(e)= 0$ and $Q(\bar e)=-b_e$, hence by \eqref{biauzzo}  $  e$ is not in the support of the $\mc C_i$'s and   $ \sum _{i=1}^n a_i \mathds{1}(\bar e \in \mc C_i)=-b_e$.  This implies \eqref{tagliamento}.
      \end{proof}

\subsection{LDP  for the  homological coefficients}
 Given a  cycle $\mc C$ in $G$, its \emph{affinity} $\mc A( \mc C)$ is  defined as (cf. \cite{SCH})
 \begin{equation}
 \mc A( \mc C):= \sum_{j=1} ^k \log  \frac{ r(x_j, x_{j+1})}{r(x_{j+1},x_j)}= \sum_{j=1}^k \log 
   \frac{\pi(x_j)  r(x_j, x_{j+1})}{\pi(x_{j+1} )r(x_{j+1},x_j)}
 =
 \sum_{j=1} ^k w_\pi ( x_j,x_{j+1})\,,
 \end{equation}
 where $ \mc C=(x_1,x_2, \dots, x_k)$
 and $w_\pi:E \to \bb R$ is the function defined in \eqref{wpi}. Note that we can also write \[ \mc A(\mc C)= \frac{1}{2} \sum _{e \in E} S_e(\mc C) w_\pi (e)\,,\]
 hence the above affinity induces a linear map on the cycle space $\mc V_*$.

 \smallskip

 From now on we fix a spanning  tree $\mc T$ in $G_u=(V,E_u)$ and chords $\mathfrak{c}'s$.
Given distinct elements $ y\not = z$ in $V$, we call $\gamma_{y,z}$  the unique self--avoiding path 
 $y=y_1,y_2, y_3, \dots , y_n=z$   from $y$ to $z$ in the tree  $\mc T$.

 Finally we  come back to our Markov chain.   To the
 trajectory read up to time $T$, $(X_t)_{0\leq t \leq T}$, we associate the  cycle $\mc C_T$ as  follows.
Let $X_0=x_1, x_2, \dots, x_n=X_T$ be the states visited by the path  $(X_t)_{0\leq t \leq T}$, chronologically ordered. If $X_T=X_0$,   then we set  $\mc C_T:=( x_1, x_2, \dots, x_n)$. If $X_T\not=X_0$,  then $\mc C_T:=( x_1, x_2, \dots, x_n, y_2, \dots, y_m)$ where $(x_n, y_1, \dots, y_m)$ is the canonical path $\gamma _{X_T,X_0}$. Roughly speaking the cycle $\mc C_T$ is obtained by gluing the trajectory $(X_t)_{0\leq t \leq T}$ with the canonical path $\gamma _{X_T,X_0}$ and then keeping knowledge only of the visited sites (disregarding the jump times).

\medskip

Enumerating the chords as $\mathfrak{c}_k$, $k \in K$, we consider   the  fundamental basis $ \mc C_k $, $k \in K$,  where $\mc C_k:= \mc C _{\mathfrak{c}_k}$.
For  each $k \in  K$  and $T \geq 0$   we define the \emph{ empirical homological coefficient }  $ a_T(k)$ as the map
$a_T(k): D( \bb R_+, V) \mapsto \bb R$ characterized by the identity in $\mc V_*$
\begin{equation}\label{sprint}
 \mc C_T[X] = \sum _{k \in K} T a_T (k) [X] \mc C_k\,,
 \end{equation} where 
 $X=(X_t)_{t \in \bb R_+}$ and  
$\mc C_T[X]$ denotes the cycle associated to the trajectory $(X_t)_{t \in [0,T]}$.   Note that we can think of $a_T$ as a map $a_T: D( \bb R_+, V) \to L_1(K)$. We endow $L_1(K)$ with the bounded weak* topology. When $K$ is finite this  reduces to the standard $L_1$--topology. We write $ \underline{a}= \bigl( a(k)
\,:\, k \in K\}$ for a generic element of $L_1(K)$.

 Before stating our LDP for $a_T$ we give a representation result. To this aim,
  for each $k$,  let $J_k$ be the current  in $L^1_\mathrm{a}(E)$ satisfying $ J_k(e):= S_e (\mc C_k) $ for all $e \in E$.

\begin{lemma}\label{bambini} If $J \in
L^1_\mathrm{a}(E)$ has zero divergence, then $ J= \sum _{k} J(\mathfrak{c}_k ) J_k$ pointwise:  
$ J(e)= \sum _{k} J(\mathfrak{c}_k ) J_k(e)$ for all $e \in E$ and the series $ \sum _{k} J(\mathfrak{c}_k ) J_k(e)$  is absolutely convergent for all $e \in E$.
\end{lemma}
\begin{proof} Let $Q(e):= [ J(e) ]_+$ for any $e \in E$, where $[x]_+:= \max \{x,0\}$. Then $Q\in L^1 (E)$ and $\div Q=0$. By Lemma 4.1 in \cite{BFG} we can write $Q= \sum _{\mc C} \alpha _{\mc C} \mathds{1}_{\mc C}$ with $\alpha _{\mc C} \geq 0$ and $\mc C$ varying among  the self--avoiding cycles (the function  $\mathds{1}_{\mc C}:E \to \{0,1\}$ is defined as  $\mathds{1}_{\mc C}(e):=\mathds{1}(e \in \mc C$)). Since
$J(e)= Q(e)- Q( \bar e) $ we have
$$ J(e)=  \sum _{\mc C} \alpha _{\mc C} \mathds{1} ( e \in \mc C)- \sum _{\mc C} \alpha _{\mc C} \mathds{1} (\bar e \in \mc C)\,,$$
and both series in the r.h.s. are  convergent (and therefore absolutely convergent). Hence we can arrange the terms as we prefer and get  the identities
\begin{equation}\label{soldi} J(e)=  \sum _{\mc C} \alpha _{\mc C}  \bigl(  \mathds{1} ( e \in \mc C)-\mathds{1} (\bar e \in \mc C))=  \sum _{\mc C}   \alpha _{\mc C} S_e(\mc C)\,,
\end{equation}
and the above series in \eqref{soldi} are absolutely convergent.
 By \eqref{lampone} we can write
$
\mc C = \sum _{k } S_{ \mathfrak{c}_k   } \bigl( \mc C \bigr)  \mc C _k$, which is 
 indeed  a finite sum. In particular, 
$ S_e(\mc C)$ is given by the finite sum $ \sum _{k } S_{ \mathfrak{c}_k  } \bigl( \mc C \bigr)  S_e(\mc C _k)$.
Coming back to \eqref{soldi} we get
\begin{equation}\label{silenzioso}J(e) =  \sum _{\mc C}   \alpha _{\mc C} \Big( \sum _{k } S_{ \mathfrak{c}_k   } \bigl( \mc C \bigr)  S_e(\mc C _k) \Big)\,.\end{equation}
Since $S_e( \mc C_k)\in \{0,-1,1\}$ we can bound (recall that $\mc C$ is self--avoiding)
\begin{equation*}
\begin{split}
\sum _{\mc C}
  \sum _{k }  | \alpha_{\mc C}  S_{ \mathfrak{c}_k  } \bigl( \mc C \bigr)  S_e(\mc C _k) |
 & \leq
 \sum _{\mc C}    \sum _{k }  \alpha_{\mc C}  |S_{ \mathfrak{c}_k   } \bigl( \mc C \bigr)  |\\
 &=\sum _{\mc C}    \sum _{k }  \alpha_{\mc C}   \bigl(  \mathds{1} ( \mathfrak{c}_k   \in \mc C)+\mathds{1} ( \bar {\mathfrak{c}_k}   \in \mc C))\\
 &
=\sum_k \sum _{\mc C}    \alpha_{\mc C}  \mathds{1} ( \mathfrak{c}_k    \in \mc C)
 +\sum_k \sum _{\mc C}    \alpha_{\mc C}  \mathds{1} ( \bar{\mathfrak{c}_k}   \in \mc C) \\
 &=\sum _k Q(  \mathfrak{c}_k )+ \sum _k Q(  \bar{\mathfrak{c}_k}) \leq  \|Q\|_1<+\infty\,.
\end{split}
\end{equation*}
Hence the series in \eqref{silenzioso} is absolutely convergent, and we can rearrange its terms getting  the following identities concerning absolutely convergent series (recall \eqref{soldi}):
$$ J(e)= \sum_k  S_e(\mc C _k)  \Big( \sum _{\mc C}   \alpha _{\mc C} S_{ \mathfrak{c}_k   } \bigl( \mc C \bigr)  \Big)= \sum_k  S_e(\mc C _k) J(\mathfrak{c}_k )= \sum_k J_k(e) J(\mathfrak{c}_k ) \,. \qedhere
$$
\end{proof}

\bigskip 
Due to \eqref{deltino} and \eqref{lampone} it holds \[
a_T (k) =\frac{1}{T} S_{\mathfrak{c}_k  } (\mc C_T)  = J_T( \mathfrak{c}_k)\,.
\]
Indeed,
since  $\gamma _{y,z}$ is the only self--avoiding  path from $y$ to $z$   inside the spanning tree $\mc T$, we have  $S_{\mathfrak{c}_k } \left(\, \gamma_{X_T,X_0} \,\right)= 0$. In conclusion, \begin{equation}\label{chiave}
 \{a_T (k ) : k \in K\}= \{ J_T( \mathfrak{c}_k) \,:\, k \in K\}\,.
 \end{equation}

We have now all the tools to prove the following result (recall the definition of $J_k \in L^1_\mathrm{a}(E)$ given before Lemma \ref{bambini}):
\begin{theorem}\label{anonimo}
 Assume the Markov chain satisfies (A1),(A2), (A3) and  Condition $C(\sigma)$ with $\sigma >0$.
    Then  the following holds:

    \begin{itemize}
    \item[(i)] As
  $T\to+\infty$ the sequence of probability measures $\{\bb P_x\circ
  a_T^{-1}\}$ on $ L^1(K)$ (endowed with the bounded weak* topology)
  satisfies a large deviation principle  with good and convex rate function $ I_c:L^1(K) \to [0,+\infty]$ such that
  \begin{equation}\label{ninnananna}
  I_c( \underline{a} ) =
  \begin{cases}
     \hat I( \sum _{k\in K} a_k J_k ) & \text{ if }   \sum_e |\sum _{k\in K} a_k J_k(e)|<+\infty\,,\\
  +\infty & \text{othewise}\,,
  \end{cases}
   \end{equation}
   where $\hat I$ is the good and convex rate function of the LDP for  the empirical current obtained  by contraction from Theorem \ref{LDP:misura+cor}.

\item[(ii)]
  Suppose in addition that   the function $w_\pi $ introduced in \eqref{wpi} is in  $C_0(E)$. If
  $\sum_e |\sum _{k\in K} a_k J_k(e)|=+\infty$, then  it holds  $I_c( \underline{a} )= I_c( -\underline{a} )=+\infty$; otherwise
  in $\bb R \cup \{ +\infty\}$ it holds
\begin{equation}\label{bigne1}
  I_c( \underline{a} )=I_c(- \underline{a} )  - \sum_e \Big(\sum _{k\in K} a_k J_k(e)\Big)w_\pi(e)\,.
\end{equation}

 \medskip  When   $\sum_{e \in E}  \sum _{k\in K} |a_k J_k(e)|<+\infty$, then \eqref{bigne1} can be rewritten as
  \begin{equation}\label{bigne2}
    I_c( \underline{a} )=I_c(- \underline{a} ) -\sum_k a_k \mc A( \mc C_k)\,,
    \end{equation}
    and the last series is indeed absolutely convergent.

\item[(iii)]   If
   the conditions of Theorem \ref{freedom} are satisfied, then the above results remain true with $L^1(K)$ endowed with the strong $L^1$--topology and $w_\pi$  bounded.

\end{itemize}
\end{theorem}

Some comments on the above theorem:

\bigskip

{\bf Comment 1}.
Since $J_k(e)= \mathds{1}( e\in \mc C_k) -  \mathds{1}(\bar e\in \mc C_k)$ and the cycle $\mc C_k$ is self--avoiding, we have $J_k(e) \in \{-1,0,1\}$ and therefore
\begin{equation}\sum _{k\in K} |a_k J_k(e)|<+\infty \qquad \forall e \in E \end{equation}
 for any $\underline{a}\in L^1(K)$. Hence the map $E \ni e \mapsto J(e):=
 \sum _{k\in K} a_k J_k  (e)\in \bb R $ is well defined (indeed the r.h.s. is absolutely convergent) and antisymmetric. If  in addition $\sum_e |\sum _{k\in K} a_k J_k(e)|<+\infty$ then the above $J$ belongs to $L^1(E)$, hence   $J$ is a summable current in $L^1_\mathrm{a}(E)$.

\bigskip

{\bf Comment 2}.
 We have
\begin{equation}
\sum_{e \in E}  \sum _{k\in K} |a_k J_k(e)|= \sum _{k\in K} |a_k|  \ell (\mc C_k)
\end{equation}
where  $\ell (\mc C_k) $ denotes the number of edges in $\mc C_k$. Hence, the condition leading to \eqref{bigne2} can be rewritten as $\sum _{k\in K} |a_k|  \ell (\mc C_k)<+\infty$.
It is therefore useful to know if a graph admits a fundamental basis whose cycles have uniformly  bounded length. Only some partial results in this direction have been achieved in graph theory\cite{diestel}. For example, working with the field $\bb F_2$ instead of $\bb R$, the following result is proved in \cite{MH}: if a    locally finite transitive\footnote{A graph $G$ is called \emph{transitive}  if for any two vertices $v,w$ one can exhibit a graph automorphism of $G$ mapping $v$ to $w$}   graph has the property that the cycle space is generated by cycles of uniformly bounded length, then the graph must be  accessible (we refer to \cite{MH} for the terminology).  

The lattice  $\bb Z^d$ does not admit a 
 fundamental basis whose cycles have uniformly  bounded length (see the appendix). 
 Positive examples can be easily constructed.

 \bigskip

{\bf Comment 3}.
If ${\sum}_{e \in E} { \sum }_{k\in K} |a_k J_k(e)|<+\infty$ then $J$
  has zero divergence. Indeed, given $y \in V$ we have (in the third identity we use that the series is absolutely convergent)
\begin{equation*}
\begin{split}
 \div J&= \sum _{z } J(y,z)= \sum _z \Big(  \sum _{k\in K} a_k J_k  (y,z)\Big)\\
 &=  \sum _{k\in K} a_k \Big( \sum _z J_k  (y,z) \Big) = \sum _{k\in K} a_k\cdot 0 =0\,.
 \end{split}
\end{equation*} 
 
 \bigskip

{\bf Comment 4}. When working with finite graphs, one can deal  with  an arbitrary basis of the cycle space, fixing arbitrarly once and for all the paths $\gamma _{y,z}$ from $y$ to $z$ and defining the homological coefficients as in \eqref{sprint}  referred to the chosen basis. Then the above theorem remains true and \eqref{bigne2} is always satisfied (cf. \cite{FD}).

\bigskip

{\bf Comment 5}.
 For infinite graphs $G$, the LDP stated in Theorem \ref{anonimo} refers to the coefficients in a given basis of $H_1(G,\bb R)$ and is not intrinsic to $H_1(G,\bb R)$. Indeed, for suitable graphs $G$, by choosing different fundamental trees $\mc T_1$, $\mc T_2$ and associated fundamental cycle bases $\{\mc C_k^{(1)}: \, k \in \bb N_+ \}$, $\{ \mc C_k^{(2)}:\, k\in \bb N_+ \}$,
 one can exhibit a sequence of cycles $\bigl(\mc C_n\bigr)_{n \geq 1}$ with the following property: 
 setting  $\mc C_n= \sum _{k =1}^\infty a_k^{(n)} \mc C_k^{(1)} = \sum _{k =1}^\infty b_k^{(n)} \mc C_k^{(2)}$, the $n$--sequence 
 $(a_k^{(n)}\,:\, k \in \bb N_+)_{n \geq 1 }$ does not converge in $L^1(\bb N_+)$ (endowed with the bounded weak* topology), while the $n$--sequence   $(b_k^{(n)}\,:\, k \in \bb N_+)_{n\geq 1}$ does. See Figure \ref{contrario} where $(a_k^{(n)}\,:\, k \in \bb N_+)$ is the string $ (1,1, \dots, 1 , 0,0 \dots)$   with $n$ 1's, while $(b_k^{(n)}\,:\, k \in \bb N)$ is the string $(0,0, \dots,0,1,0,0 ,\dots)$ with a single 1 located at position $n$ (note that $(b_k^{(n)}\,:\, k \in \bb N)$ converges to the zero element of $L^1(\bb N_+)$ in the bounded weak* topology).

 \bigskip

 \begin{figure}[htbp!]
\centering
\begin{overpic}[scale=0.45]{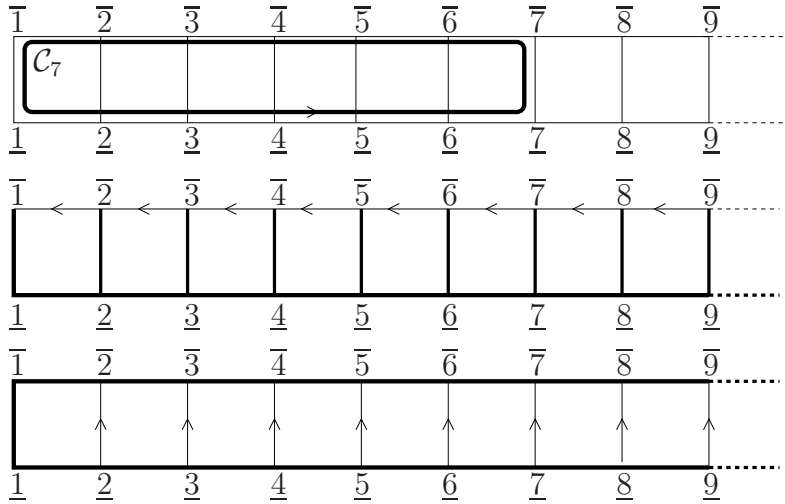}
\put(0,42){$\underline{1} $}
\put(11.2,42){$\underline{2} $}
\put(22.4,42){$\underline{3} $}
\put(33.6,42){$\underline{4} $}
\put(44.4,42){$\underline{5} $}
\put(55.7,42){$\underline{6} $}
\put(66.9,42){$\underline{7} $}
\put(78.1,42){$\underline{8} $}
\put(89.3,42){$\underline{9} $}
\put(0,57){$\overline{1} $}
\put(11.2,57){$\overline{2} $}
\put(22.4,57){$\overline{3} $}
\put(33.6,57){$\overline{4} $}
\put(44.4,57){$\overline{5} $}
\put(55.7,57){$\overline{6} $}
\put(66.9,57){$\overline{7} $}
\put(78.1,57){$\overline{8} $}
\put(89.3,57){$\overline{9} $}
\put(0,19){$\underline{1} $}
\put(11.2,19){$\underline{2} $}
\put(22.4,19){$\underline{3} $}
\put(33.6,19){$\underline{4} $}
\put(44.4,19){$\underline{5} $}
\put(55.7,19){$\underline{6} $}
\put(66.9,19){$\underline{7} $}
\put(78.1,19){$\underline{8} $}
\put(89.3,19){$\underline{9} $}
\put(0,34.6){$\overline{1} $}
\put(11.2,34.6){$\overline{2} $}
\put(22.4,34.6){$\overline{3} $}
\put(33.6,34.6){$\overline{4} $}
\put(44.4,34.6){$\overline{5} $}
\put(55.7,34.6){$\overline{6} $}
\put(66.9,34.6){$\overline{7} $}
\put(78.1,34.6){$\overline{8} $}
\put(89.3,34.6){$\overline{9} $}
\put(0,-3){$\underline{1} $}
\put(11.2,-3){$\underline{2} $}
\put(22.4,-3){$\underline{3} $}
\put(33.6,-3){$\underline{4} $}
\put(44.4,-3){$\underline{5} $}
\put(55.7,-3){$\underline{6} $}
\put(66.9,-3){$\underline{7} $}
\put(78.1,-3){$\underline{8} $}
\put(89.3,-3){$\underline{9} $}
\put(0,12.5){$\overline{1} $}
\put(11.2,12.5){$\overline{2} $}
\put(22.4,12.5){$\overline{3} $}
\put(33.6,12.5){$\overline{4} $}
\put(44.4,12.5){$\overline{5} $}
\put(55.7,12.5){$\overline{6} $}
\put(66.9,12.5){$\overline{7} $}
\put(78.1,12.5){$\overline{8} $}
\put(89.3,12.5){$\overline{9} $}
\put(3,52){$\mc C_7$}
\end{overpic}
\bigskip
\caption{\emph{Top}. The graph $G$ is  the ladder with vertex set $\{\underline{1}, \underline{2}, \dots\}\cup \{\overline{1}, \overline{2} , \dots \}$,  cycle $\mc C_n$ is given by  $(\underline{1}, \underline{2}, \dots, \underline{n}, \overline{n}, \overline{n-1}, \dots, \overline{1}$).  \emph{Center}. The fundamental tree $\mc T^{(1)}$ is the bold comb. The arrows correspond to  the chords. The basis cycle $\mc C^{(1)}_k$ is given by $(\overline{k+1}, \overline{k},\underline{k}, \underline{k+1})$. \emph{Bottom}.  The fundamental tree $\mc T^{(2)}$ is in boldface.  The basis cycle $\mc C^{(2)}_k$  equals $\mc C_k$. }
\label{contrario}
\end{figure}

\begin{proof} Item (i) as well as the first part of Item (ii)  are a  consequence of Lemma \ref{bambini}, identity \eqref{chiave} and  the LDP for the empirical current obtained  by contraction from Theorem \ref{LDP:misura+cor}.     Let us now prove  \eqref{bigne2} when $\sum_{e \in E}  \sum _{k\in K} |a_k J_k(e)|<+\infty$. By Item (i)   we have   $I_c( \underline{a} )=\hat I (J) $ and $I_c( -\underline{a} ) = \hat I(-J)$, where
$J= \sum _{k\in K} a_k J_k$ (which is indeed a  summable current with zero divergence due to Comment 3 above).
Due to \eqref{maghetto} we then have
\[ I_c( \underline{a} )=I_c(- \underline{a} ) - \frac{1}{2} \langle J, w_\pi \rangle = I_c(- \underline{a} ) - \frac{1}{2} \sum _e\Big (\sum _{k} a_k J_k(e)\Big)w_\pi(e)\,.
\]
Since $w_\pi$ is bounded and since $\sum_{e \in E}  \sum _{k\in K} |a_k J_k(e)|<+\infty$, the last series is absolutely convergent and we can rearrange it as
\[\frac{1}{2} \sum _e\Big (\sum _{k} a_k J_k(e)\Big)w_\pi(e)
= \frac{1}{2} \sum _{k} a_k \sum _e\Big( J_k(e) w_\pi(e) \Big) = \sum_k a_k \mc A(\mc C_k)\]
 By the same observations we also have
 $$ \sum_k |a_k \mc A(\mc C_k)| \leq \frac{1}{2} \sum _{k} \bigl| a_k \sum _e\left ( J_k(e) w_\pi(e) \right) \bigr| \leq \frac{\|w_\pi\|_\infty}{2} \sum _{k}\sum_e \bigl| a_k J_k(e) \bigr| <+\infty\,,$$
thus proving our thesis.

Finally, Item (iii) follows from the previous items and from Theorem \ref{freedom}.

\end{proof}


\section{Examples}
\label{s:ex}

\subsection{Markov chain with two states}

We start by the simplest possible situation: a Markov chain with
two states (a similar analysis is given in \cite{Maes2}). Let $0$ and $1$ be the two states, and denote by
$r_0=r(0,1)$ and $r_1=r(1,0)$ the corresponding jump rates. To avoid
trivialities we assume that $r_0,r_1>0$. The unique invariant  measure $\pi$ is also
reversible and is given by $\pi (0)= r_1 /(r_0+r_1)$, $\pi (1)= r_0/(r_0+r_1)$. Given $T>0$ we let $q_T :=
Q_T(0,1)+Q_T(1,0)$ be the mean  total number of jumps in the time
interval $[0,T]$. We shall here derive the large deviation principle
for the family of random variables $\{q_T\}_{T>0}$. We point out that the empirical current $J_T(0,1)$ is of order $O(1/T)$, hence the associated LDP is trivial.

By Theorem~\ref{LDP:misura+flusso} and  the contraction principle, the
family of positive random variables $\{q_T\}_{T>0}$ satisfies a large
deviation principle with rate function $f\colon\bb R_+\to \bb R_+$
given by
\begin{equation*}
  f (q) = \inf\big\{I(\mu,Q)\,:\, (\mu,Q)\in \mc P (V) \times L^+_1 (E) \;,\; Q(0,1)+Q(1,0)=q\big\}.
\end{equation*}
In view of the constraint $\div Q=0$ in \eqref{rfq} we can assume
$Q(0,1)=Q(1,0)=q/2$ and therefore
\begin{equation*}
  f(q)= \inf  \Big\{
    \Phi\big( \tfrac q2 ,\mu(0)r_0\big)+ \Phi\big(\tfrac q2 ,\mu(1)r_1\big)\,:\, \mu \in \mc P(V)
    \Big\}.
\end{equation*}
If $q=0$ we have to minimize $\mu(0) r_0+\mu(1) r_1$, getting therefore $f(0)= \min \{r_0,r_1\}$. If $q>0$, writing $\mu(0)= 1/2-\gamma$ and $\mu(1)= 1/2 + \gamma $, we need to minimize the function
$$ \psi (\gamma):=\frac{q}{2} \log \frac{q^2}{4 r_0r_1} -q + \frac{r_0+r_1}{2} + \gamma (r_1-r_0)- \frac{q}{2} \log\bigl( \frac{1}{4} - \gamma^2\bigr)
$$
over $\gamma \in [-1/2, 1/2]$.

\medskip
 Since $\psi'(\gamma)= \left[\frac{1}{4}-\gamma^2 \right] ^{-1}
\left[ (r_0-r_1) \gamma^2 + q \gamma + \frac{r_1-r_0}{4} \right]$, the optimal $\gamma$ is given by $$\left[-q+ \sqrt{q^2 +(r_0-r_1)^2} \right]/2(r_0-r_1)\,.$$
Hence the optimal $\mu$   is given by
\begin{align*}
&    \mu(0) = \frac 12 \Big( 1 +\tfrac {q}{r_0-r_1} - \frac{\sqrt{q^2 +(r_0-r_1)^2 }}{r_0-r_1}
 \Big)\,,\\
 &   \mu(1) =
 \frac 12 \Big( 1 -\tfrac {q}{r_0-r_1} +
  \frac{\sqrt{q^2 +(r_0-r_1)^2 }}{r_0-r_1}
 \Big)\,,
 \end{align*}
understanding $\mu(0)=\mu(1)=1/2$ when $r_0=r_1$.  In particular, we get
\begin{equation*}
  f(q)= \frac 12 \Big\{ q
  \log \Big[ \frac{q}{2r_0r_1}
  \bigl( \sqrt{q^2 +(r_0-r_1)^2} +q \bigr)
  \Big]
  + r_0+r_1 -q -\sqrt{q^2+(r_0-r_1)^2} \Big\}
\end{equation*}
and, in the special case $r_0=r_1=r$, $f(q)= q \log \frac{q}{r}-q+r$, which  coincides with  the
rate function of $N_T/T$ where $N_T$ is a Poisson process with
intensity $r$.
Set $\upbar{q}:= 2r_0r_1/(r_0+r_1)$ and observe that, by the law of
large numbers for the empirical flow, $q_T$ converges in probability
to $\upbar{q}$. It is simple to check that $f$ is a uniformly convex
function which achieves its minimum, as it must be the case, for
$q=\upbar{q}$.

\subsection{A random watch}
We consider the following random watch in which an hour consists of
$n$ minutes.  At time $t=0$ the minute hand is at $0$, it stays
there for an exponential time of parameter $r_0$ then it moves at
$1$, \dots, it stays at $n-1$ for   an exponential time of
parameter $r_{n-1}$ then it moves to $0$ and the hour hand advances
by one, \dots (the exponential times are all independent). Observe
that for $n>2$ the chain just defined is not reversible while for $n=2$ one recovers the previous 2 states Markov chain. The above random watch can be thought of also as a totally asymmetric random walk  on a ring with site disorder.

 Let $\mc
N_T$ be the number of hours marked by such a watch in the time
interval $[0,T]$. Taking the discrete torus $\bb T_n=\bb Z/n\bb Z$ as state space, note that $\mc N_T=\lfloor \sum _{i=0}^{n-1} TQ_T(i,i+1)/n \rfloor$, $\lfloor x \rfloor $ denoting the integer part of $x$. Hence
$\mc N_T/T$ satisfies the same large deviation principle of  $\sum _{i=0}^{n-1} Q_T(i,i+1)/n$.
  In particular, by using Theorem~\ref{LDP:misura+flusso}  and the contraction principle, we can
compute the large deviation rate function $f$ for $\mc N_T/T$. Since  the only divergence--free flows are the constant flows,
 the rate function $f : \bb R_+ \to \bb R_+$ is given by
\begin{equation}\label{matilde}
 f(q)= \inf\left \{ \sum _{i=0}^{n-1} \Phi( q, \mu_i r_i ) \,:\,   \mu_i \geq0\,, \;  \sum _{i=0}^{n-1} \mu_i=1  \right\}\,.
 \end{equation}
Trivially, $f(0)= \min \{r_i\,:\, 0 \leq i \leq n-1\}$. Let us assume $q>0$.
 Since  $\sum _{i=0}^{n-1} \Phi( q, \mu_i r_i ) \geq q \sum_{i=0}^{n-1}   \log \frac{1}{\mu_i} -C$ for a suitable constant $C$ independent from $\{\mu_i\}$, we conclude that
the above infimum is indeed achieved inside the region $\{\mu_i>0\; \forall i\}$.
Introducing the  Lagrangian multiplier $\lambda$, we first look for the extremal  points of
$$ \psi \left ( \{\mu_i\}, \zeta\right)=\sum _{i=0}^{n-1} \Phi( q , \mu_i r_i )+  \lambda \left( \sum _{i=0}^{n-1} \mu_i-1\right) \,.$$
These are characterized by the system
\begin{equation}\label{germania}
\begin{cases}
-\frac{q}{ \mu_i}+r_i + \lambda=0\,,\\
\sum _{i=0}^{n-1} \mu_i=1\,.
\end{cases}
\end{equation}
We restrict   to the region $\{\mu_i >0 \;\forall i \}$ as it must be. From the first identity we get that $\lambda >- r_\mathrm{min}$ where $r_\mathrm{min} :=\min_i r_i$.
Let $R\colon
(-r_\mathrm{min},+\infty)\to (0,+\infty)$ be the strictly increasing
function defined by
\begin{equation*}
  \frac 1{R(\lambda)} = \sum_{i=0}^{n-1} \,\frac 1{r_i+\lambda}.
\end{equation*}
We denote by $R^{-1}\colon (0,+\infty) \to
(-r_\mathrm{min},+\infty)$ the corresponding inverse function. Then the unique solution of \eqref{germania} is given by $\lambda=R^{-1}(q)$ and  $\mu_i= q/(R^{-1}(q)+r_i)$. This gives also the minimizer of \eqref{matilde}. In particular,
the large deviation rate function $f\colon \bb R_+\to
[0,+\infty)$ associated to $\mc N_T/T$  is given by
\begin{equation*}
  f(q) = \sum_{i=0}^{n-1} q \log \big( 1+ \tfrac{R^{-1}(q)}{r_i}\big)
        -R^{-1}(q)
\end{equation*}
where we understand $f(0)=r_\mathrm{min}$.

\medskip

Note that the invariant measure $\pi_i$ is given by $\pi_i= r_i^{-1} / \sum _{k=0}^{n-1} r_k^{-1}$. Hence, $\mc N_T/T$ converges in probability to $ n^{-1} \sum_{i=0}^{n-1} \pi_i r_i= R(0)$. Indeed, we have
$ f( R(0))=0$ as it must be.

Finally, we point out that $J_T(i,i+1)= Q_T(i,i+1)$, hence the large deviations for the current and for the flow coincide.

\subsection{One particle on a ring}
 Consider a homogeneous simple random walk on the discrete one dimensional torus with $N$ sites $\mathbb T_N:=\frac{\mathbb Z}{N\mathbb Z}$.
The generator of the process is
\begin{equation}
 L_Nf(x)=\lambda p \bigl[f(x+1) -f(x)\bigr] + \lambda (1-p) \bigl[ f(x-1)-f(x) \bigr]\,,
 \label{gen-1d}
\end{equation}
where $x\in \mathbb T_N$, $\lambda$ is a positive parameter and $p \in [0,1]$. We are interested
in the rate function for the empirical current $J_T(x,x+1)=Q_T(x,x+1)-Q_T(x+1,x)$. By symmetry the rate function does not  depend on $x$ (we refer to \cite{Maes3} for related results).

The rate function can be computed directly since it coincides with the rate function of
$X_T/N$ where $X_T$ is a simple random walk on $\mathbb Z$ having
generator \eqref{gen-1d}. Indeed, if for example the random walk starts at $x$, $\lfloor X_T/N \rfloor$ corresponds to the  number of  cycles made by the walker, with the rule that a clockwise cycle has weight $1$ and a unclockwise cycle has weight $-1$. In particular,  $TJ_T(x,x+1)$ differs from  $\lfloor X_T/N \rfloor$ by at most one, hence $|J_T(x,x+1) -X_T/NT| \leq 2/T$. Note that for $p=\pm 1$, we have $X_T= \pm N_T$, $(N_t)_{t \in \bb R_+}$ being a Poisson process of parameter $\l$. To simplify the treatment below, we restrict to $p\in (0,1)$ excluding the trivial cases $p=\pm 1$.

The rate function of  $X_T/N$  can be easily computed by means of  G\"artner-Ellis Theorem using the representation
$$
X_T=\sum_{i=1}^{N_T}Y_i\,,
$$
where $(N)_{t \in \bb R_+}$ is a Poisson process of parameter $\lambda$ and $Y_i$ are independent i.i.d.
random variables taking values $ 1,-1$ with  probability $p,1-p$, respectively.
We have
that the corresponding rate function $W_N$ is obtained as
\begin{equation}
W_N(j)=\sup_{\alpha\in \bb R} \left\{j\alpha-\Lambda_N(\alpha)\right\}\,,
\label{morbistenza}
\end{equation}
where
\begin{equation}
\begin{split}
 \Lambda_N(\alpha)&=\lim_{T\to +\infty}\frac 1T\log \mathbb E\left(e^{\frac{\alpha X_T}{N}}\right)\\
&  =\lim_{T\to +\infty}\frac 1T\log\left(\sum_{k=0}^{+\infty}e^{-\lambda T}
\frac{\left(\lambda T\right)^k}{k!}\mathbb E\left(e^{\frac \alpha N (Y_1+\dots +Y_k}\right)\right)\\
& =\lim_{T\to +\infty}\frac 1T\log \left(\sum_{k=0}^{+\infty}e^{-\lambda T}
\frac{\left(\lambda T\right)^k}{k!}\left( p e^{\frac \alpha N}+(1-p)e^{-\frac \alpha N} \right)^k\right)\\
& =\lambda \,p \,e^{\frac \alpha N}+\lambda (1-p) e^{-\frac \alpha N} -\lambda\,.
\label{sciop}
\end{split}
\end{equation}
Putting \eqref{sciop} into \eqref{morbistenza}, one gets that the  supremum in \eqref{morbistenza} is attained at $\alpha=N \log\left( Nj/2 p \lambda+   (1/2p\lambda) \sqrt{ (Nj)^2+4p(1-p)\lambda^2}\right)$, hence
\begin{align}
W_N(j)=& Nj\log\Big(\frac{Nj}{2p\lambda}+\frac{1}{2p\lambda}\sqrt{ (Nj)^2+4p(1-p)\lambda^2}\,\Big)\nonumber\\& -
\sqrt{(Nj)^2+4p(1-p)\lambda^2}+\lambda\,.
\label{ilratetoro}
\end{align}
Note that,  when $p=1/2$ and  $\lambda=\lambda_N=\gamma N^2$ (diffusive rescaling), it holds
$$
\lim_{N\to +\infty}W_N(j)=\frac {j^2}{2\gamma}\,,
$$
in agreement with formula (58) in \cite{Maes1} for the large deviation rate function 
for  the current of a diffusion on the circle. 
\smallskip

The same result, i.e.\ the LD rate functional for $J_T(x,x+1)$,  can be obtained by  a purely variational approach.   We write $J$ for the unique 
 zero divergence   current  such that $J(x,x+1)= j$. By Theorem \ref{LDP:misura+cor} and the contraction principle, we get
 $$ W_N(j) =   \inf \left\{\tilde I (\mu,J ) \,:\,\mu \in \mc P (V)\right\}\,,
 $$ 
 where $\tilde I (\mu,J )$ has been defined in \eqref{rff_bis}.  Since $\tilde I (\cdot,J )$
is l.s.c. on the compact space $\mc P(V)$, the above infimum is obtained at some  minimizer.  
We call $\Gamma $ the   set of minimizers $ \mu \in \mc P (V)$ and observe that $\Gamma $ is convex since $\tilde I$ is convex.  As $\tilde I(\cdot, J)$ is left invariant by the  transformation  $
 \mu \to \mc T \mu$ with $\mc T \mu=\{ \mu_{y+1}\}_{y \in \bb T_N}$ if $\mu= \{\mu_y\}_{y \in \bb T_N}$,
  also $\Gamma$ is $\mc T$--invariant.   Fix $\mu \in \Gamma$. Then, $\mu, \mc T \mu, \mc T^2 \mu, \dots , \mc T^{N-1} \mu$ all belong to $\Gamma$. By convexity of $\Gamma$,  the uniform measure $\mu_*= \frac{1}{N} \sum _{j=0}^{N-1} \mc T^j \mu$ is in $\Gamma$.
 Hence, $W_N(j)= \tilde I( \mu_*,J)$ and from \eqref{rff_bis} one recovers 
 \eqref{ilratetoro}.

\subsection{Birth and death chains}\label{ex:bd}

Consider the birth and death Markov chain on $\bb Z_+=\{0,1,2,\dots
\}$ with rates $r(k,k+1)= b_k>0$ for $k \geq 0$ and $r(k,k-1) = d_k>0$ for $k
\geq 1$. This chain has been treated in details in \cite{BFG}.  Here
we restrict to investigate when the joint LDP for the empirical
measure and flow holds with the $L^1$--strong topology instead of the
bounded weak* topology.

As proved in \cite{BFG}[Sec. 9], if $\lim_{k \to \infty} d_k=+\infty$
and $\varlimsup _{k \to \infty}b_k/d_k < 1$, then Condition
$C(\sigma)$ is satisfied for some $\sigma>0$ (as well the basic
assumptions (A1),...,(A4)). Then the following holds

\begin{proposition}\label{covalent} 
  Suppose that $\lim_{k \to \infty} d_k=+\infty$ and $\varlimsup _{k
    \to \infty}b_k/d_k < 1$.
\begin{itemize}
\item[(i)]If $\lim_{k\to \infty} b_k/d_k =0 $, then the joint LDP for
  $(\mu_T, Q_T)$ holds with $L^1_+(E)$ endowed with the strong topology;
\item[(ii)] If $\varliminf_{k\to \infty} b_{k}/d_k >0 $, then the joint
  LDP for $(\mu_T, Q_T)$ does not hold with $L^1_+(E)$ endowed with the
  strong topology.
\end{itemize}
\end{proposition}

\begin{proof}
  We first derive (i) by applying Theorem \ref{freedom} to which we
  refer for the notation.  
  We define $\hat E := \{ (k,k+1)\,:\, k \in \bb Z_+\}$.  Then $H(k) =
  b_k/(b_k+d_k)$ so that, by assumption, $\lim_{k \to \infty} H(k) =
  0$. Hence, Items (i) and (ii) of Theorem \ref{freedom} are
  satisfied. 

  Given $a>0$ and a state $x\in \bb Z_+$ , we choose $k_*\ge x$ such
  that $H(k) < a$ for any $k\ge k_*$ and define 
  $W=W(x,a): = \{ (k,k+1), \; (k+1,k) \,,\, k \ge k_*\}$. In
  particular, Items (iii.1) and (iii.2) in Theorem \ref{freedom} are
  satisfied. 

  It remains to check Item (iii.3). For any path exiting from $x$,
  given $k \geq k_*$ we get that the number of times the path uses
  the edge $(k,k+1)$ is at least the number of times the path uses the
  edge $(k+1,k)$
  (more precisely, we have equality when the path ends inside
  $[0,k]\cap \bb Z_+$
  while we have a difference of one unit if the path ends outside
  $[0,k]$). In conclusion \eqref{broccolo} is valid with
  $\gamma=1/2$.

\smallskip
To prove Item (ii) we generalize the argument used at the end of
Section 9 in \cite{BFG}. We restrict to $n $ large enough that
$1/d_n+1/d_{n+1}<1$. In this case we define
 \begin{equation*}
  \begin{split}
   \gamma_n & := 1-\frac{1}{d_n}-\frac{1}{d_{n+1}},\\
   \mu^n & := \gamma_n \, \pi
             +\frac{\delta_n}{d_n} +\frac{\delta_{n+1}}{d_{n+1} },
    \\
    Q^n  & := \gamma_n Q^\pi
             +   \delta_{(n,n+1)} +\delta_{(n+1,n)} \,.
  \end{split}
\end{equation*}
Note that $Q^n$ is divergence--free.  For all edges $(y,z)$ different
from $(n,n-1),(n,n+1), (n+1,n),(n+1,n+2)$ it holds $\Phi\bigl( Q^n
(y,z), \mu^n (y) r(y,z)\bigr)=0$ since 
$ Q^n (y,z)= \mu^n (y) r(y,z)$.

On the other hand
\begin{align*}
&
\Phi\bigl( Q^n (n,n-1), \mu^n (n) r(n,n-1)\bigr)= \Phi \bigl( q^{(1)}_n,  p^{(1)}_n\bigr)\,,\\
&   \Phi\bigl( Q^n (n,n+1), \mu^n (n) r(n,n+1)\bigr)=   \Phi \bigl( q^{(2)}_n,  p^{(2)}_n\bigr)\,,\\
&   \Phi\bigl( Q^n (n+1,n), \mu^n (n+1) r(n+1,n)\bigr)= \Phi \bigl( q^{(3)}_n,  p^{(3)}_n\bigr)\,,\\
&   \Phi\bigl( Q^n (n+1,n+2), \mu^n (n+1) r(n+1,n+2)\bigr)=  \Phi \bigl( q^{(4)}_n,  p^{(4)}_n\bigr)\,,
\end{align*}
where
\begin{align*}
&
q^{(1)}_n:= \gamma_n  Q^\pi(n,n-1)\,,&  p_n^{(1)}:=  \gamma_n Q^\pi (n,n-1)
             + 1,\\
& q^{(2)}_n:=  \gamma_nQ^\pi(n,n+1)
  +1 \,,  & p_n^{(2)}:=  \gamma_n   Q^\pi (n,n+1) + \frac{b_n}{d_n} \,,\\
& q^{(3)}_n:=   \gamma_n Q^\pi(n+1,n) +1 \,,
&  p_n^{(3)}:= \gamma_n Q^\pi (n+1,n )+ 1\,,\\
& q^{(4)}_n:= \gamma_n   Q^\pi(n+1,n+2)
             \,,
&  p_n^{(4)}:= \gamma_n Q^\pi (n+1,n+2)+ \frac{b_{n+1}}{d_{n+1} }.
\end{align*}

Trivially, $\Phi \bigl( q^{(3)}_{n}, p^{(3)}_{n}\bigr) =0$.  For $p
\geq q$ we have $0\leq \Phi(q,p) \leq p-q$; hence $\Phi \bigl(
q^{(i)}_{n}, p^{(i)}_{n}\bigr) $ is uniformly bounded for $i=1,4$.
Since $\varliminf_{k\to \infty} b_{k}/d_k \in (0,1)$, we can extract a
subsequence $\{n_k\}_{k\geq 1}$ such that $0<  c \le b_{n_{k}}/d_{n_k}
\le c'$ for some fixed $c,c'>0$ and for all $k \geq 1$.  As $Q^\pi$ is
summable then $\gamma_n Q^\pi(n,n+1)$ is uniformly bounded.  We
conclude that
$\sup_{ k \geq 1}  \Phi \bigl( q^{(2)}_{n_k},  p^{(2)}_{n_k}\bigr) <
 +\infty$. 

 We have thus shown that $\varlimsup_{k \to \infty} I(
 \mu^{n_k},Q^{n_k}) <+\infty$. We cannot therefore have a LDP with
 $L^1_+(E)$ endowed with the strong topology since the level sets of $I$
 would be compact while the sequence $\big\{ (
 \mu^{n_k},Q^{n_k})\big\}_{k \geq 1}$ is not relatively compact in 
 $L^1_+(E)$ with the strong topology.
\end{proof}


\begin{remark}
Since the only current associated to the birth--death chain  with vanishing divergence   is the zero current, the LDP for the empirical current becomes trivial.
\end{remark}

\subsection{Random walks with confining potential and external force}
We now apply some of our previous considerations to the nearest
neighbor random walk on $\bb Z^d$  with jump rates
\begin{equation}
  \label{rld}
  r(y,z) = \exp\Big\{-\frac 12 \big[ U(z)-U(y)\big] + \frac 12 F(y,z)\Big\},
  \qquad (y,z)\in E\,,
\end{equation}
where $E:= \big\{ (y,z) \in \bb Z^d\times \bb Z^d\,, \: |x-y| =1
\big\} $,  
$U\colon \bb Z^d \to \bb R$ is a function satisfying $\sum_{y\in \bb
  Z^d } \exp\{-U(y)\}<+\infty$ (in particular $U$ has compact level
sets) and $F\in L^\infty(E)$.  It is convenient to set
\begin{equation}
  \label{rld_0}
  r_0(y,z) = \exp\Big\{-\frac 12 \big[ U(z)-U(y)\big] \Big\},
  \qquad (y,z)\in E\,.
\end{equation}
Note that when $r(\cdot, \cdot)= r_0(\cdot, \cdot)$,
 the random walk is reversible with respect to the probability
$\pi = \exp\{-U\}$, where we assume that 
$U$ has been chosen so that $\pi$ is properly
normalized. 
As usual, we denote
by $r$ the holding time parameters, i.e.\ $r(y) =\sum_{z\sim y} r(y,z)$ where the
summation is carried out over the nearest neighbors of $y$.

If one regards the random  walk with rates \eqref{rld}  as a model for the position of
a charged particle in the confining potential $U$, the function $F$
is naturally interpreted as the external field.

We start discussing explosion, i.e. Assumption (A.2).  A sufficient
condition for non explosion is given by Theorem 4.6 in \cite{Var}:
explosion does not occur if there exist a constant $\gamma \geq 0$ and
a nonnegative function $G$ such that $G(x_n)\to +\infty$ when
$r(x_n)\to +\infty$ and such that (recall \eqref{Lf})
\begin{equation}\label{nonexp}
LG(y)\leq \gamma G(y)\,,\qquad  \forall y\in \bb Z^d.
\end{equation}

 Consider the function $G(y)=e^{\frac{U(y)}{2}}$. This is nonnegative
and has compact level sets. We have
\begin{eqnarray*}
& &\sum_z r(y,z)\Big(G(z)-G(y)\Big)\leq \sum_z r(y,z)G(z)\\
& &=G(y)\sum_z e^{\frac{F(y,z)}{2}}\leq 2de^{\frac{\|F\|_\infty}{2}}G(y)\,.
\end{eqnarray*}
We therefore conclude that  explosion never occurs.

To continue our investigation of the other assumptions,
 we   consider the radial and the transversal variation of the
 potential. More precisely,  when $U \in C^1 (\bb R^d)$ we  consider the orthogonal decomposition
\begin{equation}\label{tennis} \nabla U(y)=  \langle \nabla U (y), \hat y\, \rangle \,\hat y
 + W(y) \,, \qquad y \in \bb R^d\setminus\{0\}
\end{equation}
with $\hat y:= y/ |y|$   and  $ \langle y, W(y) \rangle=0$. Above 
$\langle\cdot,\cdot \rangle$ is the inner product in $\bb R^d$.

We   say that the potential $U\in C^1(\bb R^d)$ has  \emph{diverging
  radial variation  which dominates  the transversal variation} if
\begin{equation}
\lim _{|y| \to \infty}   \langle \nabla U (y), \hat y \, \rangle=+\infty \,,\label{mela1}
\end{equation} and
 \begin{equation}
    |W(y)| \leq  \frac{\alpha}{\sqrt{d} } \langle \nabla U (y), \hat y\, \rangle + C\,,\label{mela2}
 \end{equation}
for some $\alpha\in  [0,1)$ and some $C\geq 0$.
Note that if $W$ in \eqref{tennis} is bounded, then \eqref{mela2}  is trivially satisfied with $\alpha=0$.
Moreover, note that \eqref{mela1} implies that $\lim _{|y| \to \infty} U(y)/|y|=+\infty$. 

We give a criterion assuring Condition $C(\sigma)$.

\begin{lemma}[Condition $C(\sigma)$]\label{claim2}
If  $\lim_{|y|\to\infty} r_0(y)=+\infty$, then
Condition $C(\sigma)$ holds for some $\sigma >0$. In particular, if
$U\in C^1(\bb R^d)$ has  \emph{diverging  radial variation which 
dominates the transversal variation}, 
then Condition $C(\sigma)$ holds for some $\sigma >0$. 
\end{lemma}

\begin{proof} We first prove the first part.
As $u_n$ we pick the constant
sequence $u =\exp\{U/2\big\}$. Items (i)--(iv) in Condition
$C(\sigma)$] then hold trivially.
Moreover,
\begin{equation*}
\begin{split}
  v(y) & = -\frac{Lu }{u} \,(y) = \sum_{z:z\sim y} r(y,z) -
  \sum_{z:z\sim y} \exp\big\{ \tfrac 12 \, F(y,z) \big\}
  \\
  & \ge r(y) -2d \, \exp\big\{ \tfrac 12 \|F\|_\infty \big\} \ge r_0(y)
  \exp\big\{ -\tfrac 12 \|F\|_\infty \big\} -2d \, \exp\big\{ \tfrac
  12 \|F\|_\infty \big\}
  \end{split}
\end{equation*}
which imply Items (v) and (vi).

Let now $U$ be as in the second part of the lemma.  Fix $y \in \bb Z^d
\setminus \{0\}$.  There must exist a unit vector $e \in \bb Z^d$ such
that $\langle y, e \rangle \geq |y|/\sqrt{d}$. Set $z= y-e$.  Then,
for some $\xi = y-s e$ and $s \in [0,1]$, we can write
\begin{equation*}
  \begin{split}
  U(y)-U(z) & =\langle \nabla U (\xi), e \rangle =\langle \nabla U (\xi),
\hat \xi \rangle \langle \hat \xi, e \rangle + \langle W(\xi), e
\rangle \\
& \geq \langle \nabla U (\xi), \hat \xi \rangle \left[\langle
  \hat \xi, e \rangle-\frac{\alpha}{\sqrt{d} } \right]-C\,,
\end{split}
\end{equation*}
where in the last bound we used \eqref{mela2}.
Since  $\langle \xi, e \rangle = \langle y, e\rangle - s  \geq  |y|/\sqrt{d}-1$ while $|\xi |\leq |y|+1$,   we conclude that
\begin{equation}\label{caffe}
 U(y)-U(z) \geq\frac{ \langle \nabla U (\xi), \hat \xi \rangle }{\sqrt{d}} \left[ \frac{|y|-\sqrt{d}}{|y|+1} - \alpha\right]-C\,.
\end{equation}
The above inequality gives a lower  bound for $r_0(y,z)$, and therefore for $r_0(y)$, which implies  that $\lim _{|y| \to \infty} r_0(y)=+\infty$ under assumption \eqref{mela1}.
\end{proof}


We now give a criterion assuring that the joint LDP of Theorem
\ref{LDP:misura+flusso} holds with $L^1_+(E)$ endowed with the strong
$L^1$--topology.
\begin{lemma}[LDP in $L^1$--strong topology]\label{claim3} Suppose that $U\in C^1(\bb R^d)$ has
 \emph{diverging  radial variation which dominates the transversal variation}.
 Consider one of the two following cases:

\medskip

{\sl Case 1}: $W$ is bounded (which automatically implies \eqref{mela2});

\medskip

{\sl Case 2}: \eqref{mela2} holds for some $\alpha \in [0,1/2)$ and
\begin{equation}\label{silurino}
\lim _{\substack{ |y|,|z| \to \infty\\ |y-z|\leq 1 }} \frac{ \langle \nabla U(y) , \hat y\,\rangle }{ \langle \nabla U(z) , \hat z\,\rangle}=1
\end{equation}

 \medskip

 Then, both in Case 1 and in Case 2,    Theorem \ref{LDP:misura+flusso} holds  with $L^1_+(E)$ endowed with the strong $L^1$--topology.
\end{lemma}

\begin{proof}  We apply
 Theorem \ref{freedom}   with
$$
 \hat E:=\{ (y,y+e) \in \bb Z^d \times \bb Z^d \,:\, |e |=1  \,,\; \langle y , e \rangle \geq 0 \}\,.
$$
 The validity of Item (i) of Theorem \ref{freedom} is trivial.
Let us check Item  (ii) of Theorem \ref{freedom}.
 We restrict to Case 2 (Case 1 follows the main lines and is simpler, we give some comments below).
To this aim fix $y \in \bb Z^d \setminus \{0\}$.  Take  $z \in \bb Z^d$ with $z=y+e$,  $|e|=1$ and $\langle y, e \rangle \geq 0$. Then, for some $\xi = y+s e$ and $s \in [0,1]$, we can write
$$ U(z)-U(y)=\langle \nabla U (\xi), e \rangle =\langle \nabla U (\xi), \hat \xi \rangle \langle \hat \xi, e \rangle + \langle W(\xi),
 e \rangle\,.
$$
Since $\langle \xi, e \rangle = \langle y, e\rangle + s \geq 0$, for $|y|$ large we can bound
$$ U(z)-U(y) \geq - |\langle W(\xi)| \geq - \frac{\alpha }{\sqrt{d} }  \langle \nabla U(\xi) , \hat \xi \rangle-C \,.$$
This implies for $|y| $ large  that
\begin{equation}\label{primino}
\sum _{z: (y,z) \in \hat E} r(y,z) \leq  2 d \exp \left\{   \frac{\alpha \gamma _+(y)}{2\sqrt{d} }+ \frac{\|F\|_\infty+C}{2} \right\}\,,
\end{equation}
 where
 $$ \gamma_+(y) := \sup \left\{  \langle \nabla U(\xi) , \hat \xi\, \rangle \,:\,          \xi \in \bb R^d\,,\; |\xi-y|\leq 1 \right\}\,.$$
   In Case 1 \eqref{primino} remains valid with $\frac{\alpha \gamma _+(y)}{2\sqrt{d} }$ replaced by $\sup_{i }\|W_i\|_\infty /2$.

Take $e'$ a unit vector such that $\langle y, e'\rangle \ge |y|/
\sqrt{d}$ and set $z'= z -e'$. Being in the same setting of
\eqref{caffe}, we conclude that 
  \begin{equation}\label{secondino}
    r(y,z')  \geq\exp\left\{   \frac{ \gamma_-(y) }{2\sqrt{d}} \left[
        \frac{|y|-\sqrt{d}}{|y|+1} - \alpha\right]-\frac{C}{2} -
      \frac{1}{2} \|F\|_\infty 
    \right\}\,,
 \end{equation}
 where
$$
\gamma_-(y) := \inf \left\{  \langle \nabla U(\xi) , 
\hat \xi\, \rangle \,:\,          \xi \in \bb R^d\,,\; |\xi-y|\leq 1 \right\}\,.
$$
By using \eqref{primino} and \eqref{secondino}  we get
\begin{equation}\label{cottissimi}
 H(y) \leq \frac{\sum _{z: (y,z) \in \hat E} r(y,z)}{r(y,z')} \leq C' \exp\left\{
  \frac{   \gamma_-(y)}{2 \sqrt{d} } \left(  \alpha \frac{\gamma_+(y)}{\gamma_-(y)}+\alpha 
 - \frac{|y|-\sqrt{d}}{|y|+1} 
 \right) \right\}\,.
 \end{equation}

Since the map $ \xi \to \langle \nabla U(\xi) , \hat \xi\, \rangle$ is continuous, we can write
$\gamma_+(y)=  \langle \nabla U(\xi_0) , \hat \xi_0\, \rangle$ and $\gamma_-(y)=  \langle \nabla U(\xi_1) , \hat \xi_1\, \rangle$  for suitable $\xi_0,\xi_1 $ satisfying $|\xi_0-y|, |\xi_1-y| \leq 1$. Writing
\[ \frac{\gamma_+(y)}{\gamma_-(y)}= \frac{   \langle \nabla U(\xi_0) , \hat \xi_0\, \rangle}{\langle \nabla U(y) , \hat y \, \rangle} \frac{\langle \nabla U(y) , \hat y \, \rangle}{   \langle \nabla U(\xi_1) , \hat \xi_1\, \rangle}\,,\]
by \eqref{silurino}  we deduce that $\gamma_+(y)/ \gamma_-(y)=1+o(1)$ as $|y| \to +\infty$. In particular, we can rewrite \eqref{cottissimi} as 
\[ H(y) \leq C' \exp\left\{
  \frac{   \gamma_-(y)}{2 \sqrt{d} } \bigl(  2 \alpha-1+o(1)
 \bigr)\right\}\,.\]
Using that $ \gamma_-(y) \to +\infty$ as $|y|\to +\infty$ and that $\alpha<1/2$ (we restrict to Case 2), we get Item (ii) of Theorem \ref{freedom}, i.e. that  the function $H$ defined in \eqref{defH} vanishes at infinity.
\medskip

Let us finally check Item (iii) of Theorem \ref{freedom}.
To this aim,   given a positive integer $r$,   we introduce the diamond   $B(r):=\{y \in \bb Z^d\,:\, |y|_1\leq r \}$. Given  $x \in \bb Z^d$  and $a>0$, we take $r$ large enough that $x \in B(r)$ and $\{H \geq  a\} \subset B(r-1)$ (recall that $H $  vanishes at infinity). Finally we define
$W=W(x,a)$ as the family of oriented edges  in $\bb Z^d$  not  inside  $B(r)$:
 $$W:=\{(y,z)\in E\,:\, y \not \in B(r) \text{ or } z \not \in B(r) \}\,.$$
 Trivially $W$ satisfies Items (iii.1) and (iii.2)  in Theorem \ref{freedom}.
 We
 claim that also Item (iii.3) holds:   given any path $x_1=x, x_2, x_3 \dots x_n $ of
 nearest--neighbor points in $\bb Z^d$  starting at $x$,  the number
 of its  edges   in $W \cap \hat E$ is at least $1/2$ of the
 total number of its edges  in $W$.  To prove the above
 claim it is enough to observe that, considering the pieces of the
 path in  $\{y \in \bb Z^d\,:\, |y|_1\geq r \}$, we can restrict to a
 path  $x_1, x_2, x_3 \dots x_n $  with $|x_1|_1=r$ and with $|x_i|_1
 \geq  r $ for all $i=2,\dots, n$. To prove the thesis for this path,
 we observe  that 
 $|x_{i+1} |_1=  |x_{i} |_1+1$ if $x_{i+1}-x_i \in \hat E$ while
 $|x_{i+1} |_1=  |x_{i} |_1-1$ if $x_{i+1}-x_i \not  \in \hat E$. Therefore,
 \begin{multline*} \sharp\{ i:  1\leq i <n \,,\; x_{i+1}-x _i \in \hat E \} - \sharp\{ i:  1\leq i <n \,,\; x_{i+1}-x _i \not \in \hat E \}\\ = |x_n|_1-
  |x_1|_1= |x_n|_1-r \,.
  \end{multline*}
    Since by assumption $|x_n|_1\geq r$ we get the thesis.
 \end{proof}

We next discuss some choices of the field  $F$ allowing to apply
Theorem~\ref{LDP:GC} and to deduce the large deviation principle for
the Gallavotti-Cohen functional. These hypotheses will be in the
same spirit of those introduced in \cite{bg} for continuous
diffusions. Observing that in this example it holds
$E=E_\mathrm{s}$, we restrict to the physically relevant case in which $F$ is
antisymmetric, i.e.\ $F(y,z)=-F(z,y)$, $(y,z)\in E$. We then require
that the chain with rates $r$ has the same invariant measure
$\pi=\exp\{-U\}$ as the one with rates $r_0$, that is
\begin{equation}
  \label{ort}
  \sum_{z:z\sim y} \exp\big\{-\tfrac 12 [ U(z) -U(y)] \big\}
  \, \sinh\Big( \frac 12 \, F(y,z)\Big)=0,
  \qquad \forall\: y \in V.
\end{equation}
We stress that the knowledge of $\pi$ is necessary to know the function $w_\pi$  in \eqref{wpi}, we consider here models where the external force field does not change the invariant distribution.

For simplicity  we restrict to $d=2$.
Functions $U$ and $F$ satisfying \eqref{ort} can be
easily constructed. For instance one can take $U$
``radial'', i.e.\ $U(y)=\tilde{U}(|y|_1)$ for some $\tilde{U}\colon\bb
Z_+\to \bb R$. Then the discrete vector field $F$ has to be fixed as in Fig.
\ref{fig1}. In that figure we represent the level curves of $U$ with black lines
and use arrows of different colors to represent the force field. To each color we arbitrarily associate
 a real number varying in a fixed interval $[-A,A]$
  representing the value  of the discrete vector field.
Consider an oriented edge $(y,z)$.  If in Figure \ref{fig1} there is a colored arrow from $y$ to $z$
 then $F(y,z)$ assumes the value corresponding to that color, while if there is a
 colored arrow from $z$ to $y$
 then $F(y,z)$ assumes the value corresponding to that color with a minus sign.
If there is no arrow associated either to $(y,z)$ or to $(z,y)$ then $F(y,z)=0$. Note that by construction $\|F\|_\infty$ is bounded and  \eqref{ort} is satisfied.

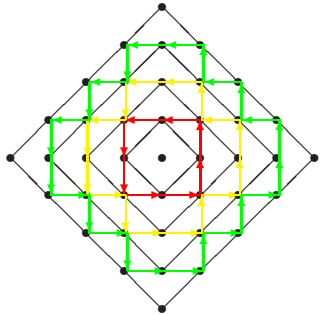
\begin{figure}
\setlength{\unitlength}{5cm}
\begin{picture}(1,1)

\put(0,0.5){\circle*{0.02}}

\put(0.1,0.4){\circle*{0.02}}
\put(0.1,0.5){\circle*{0.02}}
\put(0.1,0.6){\circle*{0.02}}

\put(0.2,0.3){\circle*{0.02}}
\put(0.2,0.4){\circle*{0.02}}
\put(0.2,0.5){\circle*{0.02}}
\put(0.2,0.6){\circle*{0.02}}
\put(0.2,0.7){\circle*{0.02}}

\put(0.3,0.2){\circle*{0.02}}
\put(0.3,0.3){\circle*{0.02}}
\put(0.3,0.4){\circle*{0.02}}
\put(0.3,0.5){\circle*{0.02}}
\put(0.3,0.6){\circle*{0.02}}
\put(0.3,0.7){\circle*{0.02}}
\put(0.3,0.8){\circle*{0.02}}

\put(0.4,0.1){\circle*{0.02}}
\put(0.4,0.2){\circle*{0.02}}
\put(0.4,0.3){\circle*{0.02}}
\put(0.4,0.4){\circle*{0.02}}
\put(0.4,0.5){\circle*{0.02}}
\put(0.4,0.6){\circle*{0.02}}
\put(0.4,0.7){\circle*{0.02}}
\put(0.4,0.8){\circle*{0.02}}
\put(0.4,0.9){\circle*{0.02}}

\put(0.5,0.2){\circle*{0.02}}
\put(0.5,0.3){\circle*{0.02}}
\put(0.5,0.4){\circle*{0.02}}
\put(0.5,0.5){\circle*{0.02}}
\put(0.5,0.6){\circle*{0.02}}
\put(0.5,0.7){\circle*{0.02}}
\put(0.5,0.8){\circle*{0.02}}

\put(0.6,0.3){\circle*{0.02}}
\put(0.6,0.4){\circle*{0.02}}
\put(0.6,0.5){\circle*{0.02}}
\put(0.6,0.6){\circle*{0.02}}
\put(0.6,0.7){\circle*{0.02}}

\put(0.7,0.4){\circle*{0.02}}
\put(0.7,0.5){\circle*{0.02}}
\put(0.7,0.6){\circle*{0.02}}

\put(0.8,0.5){\circle*{0.02}}





\put(0.4, 0.4){\line(1,1){0.1}}
\put(0.4, 0.4){\line(-1,1){0.1}}
\put(0.3, 0.5){\line(1,1){0.1}}
\put(0.5, 0.5){\line(-1,1){0.1}}

\put(0.4, 0.3){\line(1,1){0.2}}
\put(0.4, 0.3){\line(-1,1){0.2}}
\put(0.6, 0.5){\line(-1,1){0.2}}
\put(0.2, 0.5){\line(1,1){0.2}}

\put(0.4, 0.2){\line(1,1){0.3}}
\put(0.4, 0.2){\line(-1,1){0.3}}
\put(0.7, 0.5){\line(-1,1){0.3}}
\put(0.1, 0.5){\line(1,1){0.3}}

\put(0.4, 0.1){\line(1,1){0.4}}
\put(0.4, 0.1){\line(-1,1){0.4}}
\put(0.8, 0.5){\line(-1,1){0.4}}
\put(0, 0.5){\line(1,1){0.4}}

\textcolor{red}{
\put(0.4, 0.4){\vector(1,0){0.1}}
\put(0.5, 0.4){\vector(0,1){0.1}}
\put(0.5, 0.5){\vector(0,1){0.1}}
\put(0.5, 0.6){\vector(-1,0){0.1}}
\put(0.4, 0.6){\vector(-1,0){0.1}}
\put(0.3, 0.6){\vector(0,-1){0.1}}
\put(0.3, 0.5){\vector(0,-1){0.1}}
\put(0.3, 0.4){\vector(1,0){0.1}}
}

\textcolor{yellow}{
\put(0.376, 0.3){\vector(1,0){0.1}}
\put(0.276, 0.3){\vector(1,0){0.1}}
\put(0.476, 0.3){\vector(0,1){0.1}}
\put(0.476, 0.4){\vector(1,0){0.1}}
\put(0.576, 0.4){\vector(0,1){0.1}}
\put(0.576, 0.5){\vector(0,1){0.1}}
\put(0.576, 0.6){\vector(-1,0){0.1}}
\put(0.476, 0.6){\vector(0,1){0.1}}
\put(0.476, 0.7){\vector(-1,0){0.1}}
\put(0.376, 0.7){\vector(-1,0){0.1}}
\put(0.276, 0.7){\vector(0,-1){0.1}}
\put(0.276, 0.6){\vector(-1,0){0.1}}
\put(0.176, 0.6){\vector(0,-1){0.1}}
\put(0.176, 0.5){\vector(0,-1){0.1}}
\put(0.176, 0.4){\vector(1,0){0.1}}
\put(0.276, 0.4){\vector(0,-1){0.1}}
}

\textcolor{green}{
\put(0.353, 0.2){\vector(1,0){0.1}}
\put(0.453, 0.2){\vector(0,1){0.1}}
\put(0.453, 0.3){\vector(1,0){0.1}}
\put(0.553, 0.3){\vector(0,1){0.1}}
\put(0.553, 0.4){\vector(1,0){0.1}}
\put(0.653, 0.4){\vector(0,1){0.1}}
\put(0.653, 0.5){\vector(0,1){0.1}}
\put(0.653, 0.6){\vector(-1,0){0.1}}
\put(0.553, 0.6){\vector(0,1){0.1}}
\put(0.553, 0.7){\vector(-1,0){0.1}}
\put(0.453, 0.7){\vector(0,1){0.1}}
\put(0.453, 0.8){\vector(-1,0){0.1}}
\put(0.353, 0.8){\vector(-1,0){0.1}}
\put(0.253, 0.8){\vector(0,-1){0.1}}
\put(0.253, 0.7){\vector(-1,0){0.1}}
\put(0.153, 0.7){\vector(0,-1){0.1}}
\put(0.153, 0.6){\vector(-1,0){0.1}}
\put(0.053, 0.6){\vector(0,-1){0.1}}
\put(0.053, 0.5){\vector(0,-1){0.1}}
\put(0.053, 0.4){\vector(1,0){0.1}}
\put(0.153, 0.4){\vector(0,-1){0.1}}
\put(0.153, 0.3){\vector(1,0){0.1}}
\put(0.253, 0.3){\vector(0,-1){0.1}}
\put(0.253, 0.2){\vector(1,0){0.1}}
}

\end{picture}
\caption{The vector field F when $U=\tilde{U}(|x|_1)$ } \label{fig1}
\end{figure}

If instead we consider $U$ of the form $U(y)=\tilde{U}(|x|_\infty )$ for some $\tilde{U}\colon\bb
Z_+\to \bb R$ then to have \eqref{ort} we need to fix the discrete vector field $F$
as in Fig. \ref{fig2}, following the same construction as above. In both cases the discrete vector field $F$ is associated to ``rotations'' along the level
curves of $U$.

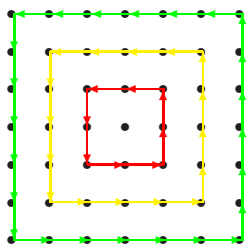
\begin{figure}
\setlength{\unitlength}{5cm}
\begin{picture}(1,1)


\put(0.1,0.2){\circle*{0.02}}
\put(0.1,0.3){\circle*{0.02}}
\put(0.1,0.4){\circle*{0.02}}
\put(0.1,0.5){\circle*{0.02}}
\put(0.1,0.6){\circle*{0.02}}
\put(0.1,0.7){\circle*{0.02}}
\put(0.1,0.8){\circle*{0.02}}

\put(0.2,0.2){\circle*{0.02}}
\put(0.2,0.3){\circle*{0.02}}
\put(0.2,0.4){\circle*{0.02}}
\put(0.2,0.5){\circle*{0.02}}
\put(0.2,0.6){\circle*{0.02}}
\put(0.2,0.7){\circle*{0.02}}
\put(0.2,0.8){\circle*{0.02}}

\put(0.3,0.2){\circle*{0.02}}
\put(0.3,0.3){\circle*{0.02}}
\put(0.3,0.4){\circle*{0.02}}
\put(0.3,0.5){\circle*{0.02}}
\put(0.3,0.6){\circle*{0.02}}
\put(0.3,0.7){\circle*{0.02}}
\put(0.3,0.8){\circle*{0.02}}

\put(0.4,0.2){\circle*{0.02}}
\put(0.4,0.3){\circle*{0.02}}
\put(0.4,0.4){\circle*{0.02}}
\put(0.4,0.5){\circle*{0.02}}
\put(0.4,0.6){\circle*{0.02}}
\put(0.4,0.7){\circle*{0.02}}
\put(0.4,0.8){\circle*{0.02}}

\put(0.5,0.2){\circle*{0.02}}
\put(0.5,0.3){\circle*{0.02}}
\put(0.5,0.4){\circle*{0.02}}
\put(0.5,0.5){\circle*{0.02}}
\put(0.5,0.6){\circle*{0.02}}
\put(0.5,0.7){\circle*{0.02}}
\put(0.5,0.8){\circle*{0.02}}

\put(0.6,0.2){\circle*{0.02}}
\put(0.6,0.3){\circle*{0.02}}
\put(0.6,0.4){\circle*{0.02}}
\put(0.6,0.5){\circle*{0.02}}
\put(0.6,0.6){\circle*{0.02}}
\put(0.6,0.7){\circle*{0.02}}
\put(0.6,0.8){\circle*{0.02}}

\put(0.7,0.2){\circle*{0.02}}
\put(0.7,0.3){\circle*{0.02}}
\put(0.7,0.4){\circle*{0.02}}
\put(0.7,0.5){\circle*{0.02}}
\put(0.7,0.6){\circle*{0.02}}
\put(0.7,0.7){\circle*{0.02}}
\put(0.7,0.8){\circle*{0.02}}


\textcolor{red}{
\put(0.4, 0.4){\vector(1,0){0.1}}
\put(0.5, 0.4){\vector(0,1){0.1}}
\put(0.5, 0.5){\vector(0,1){0.1}}
\put(0.5, 0.6){\vector(-1,0){0.1}}
\put(0.4, 0.6){\vector(-1,0){0.1}}
\put(0.3, 0.6){\vector(0,-1){0.1}}
\put(0.3, 0.5){\vector(0,-1){0.1}}
\put(0.3, 0.4){\vector(1,0){0.1}}
}

\textcolor{yellow}{
\put(0.376, 0.3){\vector(1,0){0.1}}
\put(0.476, 0.3){\vector(1,0){0.1}}
\put(0.576, 0.3){\vector(0,1){0.1}}
\put(0.576, 0.4){\vector(0,1){0.1}}
\put(0.576, 0.5){\vector(0,1){0.1}}
\put(0.576, 0.6){\vector(0,1){0.1}}
\put(0.576, 0.7){\vector(-1,0){0.1}}
\put(0.476, 0.7){\vector(-1,0){0.1}}
\put(0.376, 0.7){\vector(-1,0){0.1}}
\put(0.276, 0.7){\vector(-1,0){0.1}}
\put(0.176, 0.7){\vector(0,-1){0.1}}
\put(0.176, 0.6){\vector(0,-1){0.1}}
\put(0.176, 0.5){\vector(0,-1){0.1}}
\put(0.176, 0.4){\vector(0,-1){0.1}}
\put(0.176, 0.3){\vector(1,0){0.1}}
\put(0.276, 0.3){\vector(1,0){0.1}}
}

\textcolor{green}{
\put(0.353, 0.2){\vector(1,0){0.1}}
\put(0.453, 0.2){\vector(1,0){0.1}}
\put(0.553, 0.2){\vector(1,0){0.1}}
\put(0.653, 0.2){\vector(0,1){0.1}}
\put(0.653, 0.3){\vector(0,1){0.1}}
\put(0.653, 0.4){\vector(0,1){0.1}}
\put(0.653, 0.5){\vector(0,1){0.1}}
\put(0.653, 0.6){\vector(0,1){0.1}}
\put(0.653, 0.7){\vector(0,1){0.1}}
\put(0.653, 0.8){\vector(-1,0){0.1}}
\put(0.553, 0.8){\vector(-1,0){0.1}}
\put(0.453, 0.8){\vector(-1,0){0.1}}
\put(0.353, 0.8){\vector(-1,0){0.1}}
\put(0.253, 0.8){\vector(-1,0){0.1}}
\put(0.153, 0.8){\vector(-1,0){0.1}}
\put(0.053, 0.8){\vector(0,-1){0.1}}
\put(0.053, 0.7){\vector(0,-1){0.1}}
\put(0.053, 0.6){\vector(0,-1){0.1}}
\put(0.053, 0.5){\vector(0,-1){0.1}}
\put(0.053, 0.4){\vector(0,-1){0.1}}
\put(0.053, 0.3){\vector(0,-1){0.1}}
\put(0.053, 0.2){\vector(1,0){0.1}}
\put(0.153, 0.2){\vector(1,0){0.1}}
\put(0.253, 0.2){\vector(1,0){0.1}}
}

\end{picture}
\caption{The vector field F when $U=\tilde{U}(|x|_\infty)$} \label{fig2}
\end{figure}

The Gallavotti-Cohen functional \eqref{gcec} then becomes
\begin{equation*}
  W_T =\frac 12 \sum_{(y,z)\in E} J_T(y,z) F(y,z)
      = \sum_{y\in \bb Z^d} \sum_{i=1}^d  J_T(y,y+e_i) F(y,y+e_i)
\end{equation*}
where we used the antisymmetry of $J_T$ and $F$. In particular,
$W_T$ is naturally interpreted as the empirical power dissipated by
$F$. The  large deviation principle for the family
$\{W_T\}$ then follows from Theorem \ref{LDP:GC}. In particular, if  $F\in
C_0(E)$
 we only need to require condition $C(\sigma)$ for some $\sigma >0$ and this can be checked using the criterion given in Lemma \ref{claim2}. If  $F\in L^\infty (E)$ we need in addition to verify that the joint LDP for the empirical measure and flow holds with the $L^1$--topology instead of the bounded weak* topology for $L^1_+(E)$. This can be done by  applying Theorem \ref{freedom}, or the criterion (as well as  some variations) given in Lemma \ref{claim3}.


\appendix

\section{Geometric properties of spanning trees of $\bb Z^d$} 
We consider here the lattice $\bb Z^d$, $d \geq 2$. Trivially,  the cycle space  admits a basis given by cycles of uniformly bounded length: take the cycles $(x,x+e_i, x+e_i+e_j, x+e_j)$ where $x$ varies in $\bb Z^d$, $1\leq i < j \leq d$,  $e_i$ and $e_j$ vary among the vectors in the canonical basis of $\bb Z^d$.   Due to Comment 2 after Theorem \ref{anonimo} it is natural to ask if the lattice $\bb Z^d$  admits  
a fundamental basis   given by cycles of uniformly bounded length.  The answer is negative due to the following fact: 
\begin{proposition}
Consider a countable connected unoriented graph $\mc G= (\mc V, \mc E)$ and fix a spanning tree $\mc T$. If the fundamental cycle  basis associated to $\mc T$ has cycles with at most   $\ell+1$ vertices, then the following property holds:

Given  $a\not = b \in \mc V$ fix a path   $\gamma=(x_0,x_1, \dots, x_M)$ from $x_0=a$ to $x_M=b$.  Let 
$\gamma_{a,b}=(z_0,z_1, \dots, z_R)$ be the unique self--avoiding path inside the tree $\mc T$ from $z_0=a$ to $z_R=b$. Then for any $i: 0 \leq i \leq R$ there exists $j: 0 \leq j \leq M$ with $d(z_i,x_j) \leq \ell$, $d(\cdot, \cdot)$ being the graph distance.
\end{proposition} 

Since the property in the above proposition is trivially  not satisfied by the lattice $\bb Z^d$, $d \geq 2$,  we get that $\bb Z^d$ has no fundamental cycle basis with uniformly bounded length.
\begin{proof}
Consider the path 
  $\gamma=(x_0,x_1,x_2, \dots, x_M)$.  For each  $k =0,1, \dots, M-1$   either the edge $(x_k,x_{k+1}) $ belongs to  the tree $\mc T$, or  it is a chord and therefore the vertices $x_k,x_{k+1}$ have graph distance bounded by $\ell$ inside $\mc T$. We modify $\gamma $ as follows. If the edge  $(x_k,x_{k+1}) $ belongs to  the tree $\mc T$,  then keep the pair $x_k,x_{k+1}$ unchanged, otherwise replace the pair $x_k,x_{k+1}$ by the string $x_k, a_1, a_2, \dots, a_r,x_{k+1}$ given by  the unique self--avoiding path inside $\mc T$ from $x_k$ to $x_{k+1}$. We call $\gamma^{(1)}$ the resulting new path.  
Writing  $\gamma ^{(1)}  =(y_0,y_1, \dots, y_S)$, we get that  $y_0=a, y_S=b$, $\gamma^{(1)}$ lies inside the tree and that 

\begin{equation}\label{treno}
\forall  i: 0\leq i \leq S \qquad \exists j: 0 \leq j \leq M \text{ such that }d(y_i, x_j) \leq \ell\,.
\end{equation}
   The path $\gamma^{(1)}$ could have self--intersections,  anyway thought of as an unoriented  graph it is a connected subgraph of $\mc T$, hence it contains a self--avoiding  path from $a$ to  $b$, which (by definition of tree) must be $\gamma_{a,b}$. In particular, the vertices of $\gamma_{a,b}$ are of the form $y_i$ and therefore satisfy \eqref{treno}.

  \end{proof}

  \bigskip
  
  \noindent
  {\bf Aknowledgements}. We thank R. Diestel for useful discussions.

\end{document}